\documentclass[12pt]{amsart}
\usepackage{amsmath}
\usepackage{amsfonts}
\usepackage{amssymb}
\usepackage{xcolor}
\usepackage{bbding}
\usepackage{stmaryrd}
\usepackage{txfonts}
\usepackage{graphicx}
\usepackage{ytableau}
\usepackage{epsfig}
\usepackage{xypic}
\usepackage{tikz}
\usepackage{longtable}
\usepackage[all]{xy}
\usepackage{pgflibraryarrows}
\usepackage{pgflibrarysnakes}
\usepackage[shortlabels]{enumitem}
\usepackage{ifpdf}
\ifpdf
  \usepackage[colorlinks,final,backref=page,hyperindex]{hyperref}
\else
  \usepackage[colorlinks,final,backref=page,hyperindex,hypertex]{hyperref}
\fi
\usepackage{xspace}

\newtheorem{theorem}{Theorem}[section]
\newtheorem{lemma}[theorem]{Lemma}
\newtheorem{corollary}[theorem]{Corollary}

\newtheorem{proposition}[theorem]{Proposition}
\theoremstyle{definition}
\newtheorem{definition}[theorem]{Definition}

\newtheorem{example}[theorem]{Example}

\newcommand{\nc}{\newcommand}
\newcommand{\delete}[1]{}

\nc{\tred}[1]{\textcolor{red}{#1}}
\nc{\tblue}[1]{\textcolor{blue}{#1}} \nc{\tgreen}[1]{\textcolor{green}{#1}} \nc{\tpurple}[1]{\textcolor{purple}{#1}} \nc{\btred}[1]{\textcolor{red}{\bf #1}} \nc{\btblue}[1]{\textcolor{blue}{\bf #1}} \nc{\btgreen}[1]{\textcolor{green}{\bf #1}} \nc{\btpurple}[1]{\textcolor{purple}{\bf #1}}

\newcommand{\efootnote}[1]{}

\nc{\mlabel}[1]{\label{#1}}  
\nc{\mcite}[2][]{\cite[#1]{#2}}  
\nc{\mref}[1]{\ref{#1}}  
\nc{\mbibitem}[1]{\bibitem{#1}} 

\delete{
\nc{\mlabel}[1]{\label{#1}  
{\hfill \hspace{1cm}{\bf{{\ }\hfill(#1)}}}}
\nc{\mcite}[1]{\cite{#1}}  
\nc{\mref}[1]{\ref{#1}{{\bf{{\ }(#1)}}}}  
\nc{\mbibitem}[1]{\bibitem[\bf #1]{#1}} 
}

\renewcommand\geq{\geqslant}
\renewcommand\leq{\leqslant}

\renewcommand\bar[1]{\overline{#1}}

\nc{\nz}{\varepsilon}
\nc{\Id}{\mathrm{Id}}
\nc{\map}[2]{{#2}^{#1}}
\nc{\gp}{B}

\nc{\Irr}{\mathrm{Irr}}
\nc{\vx}{\sigma} \nc{\vy}{\tau} \nc{\dvx}{\sigma^{(1)}} \nc{\dvy}{\tau^{(1)}} \nc{\done}{\vep} \nc{\mcitep}[1]{\mcite{#1}} \nc{\wt}{\mathrm{wt}} \nc{\bre}[1]{|#1|} \nc{\mapmonoid}{\frakM} \nc{\disjoint}{\frakM'}
\nc{\ncpoly}[1]{\langle #1\rangle}  
\nc{\mapm}[1]{\lfloor\!|{#1}|\!\rfloor}
\nc{\diff}[1]{{}^\NC\{ #1 \}} \nc{\disj}[1]{\{{#1}\}'} \nc{\mdisj}[1]{\frakM'(#1)} \nc{\brho}{\bar{\rho}} \nc{\om}{\bar{\frakm}} \nc{\frakn}{\mathfrak n} \nc{\ddeg}[1]{^{(#1)}} \nc{\opset}{X} \nc{\genset}{{Z}} \nc{\NC}{\mathrm{{NC}}} \nc{\leaf}{\mathrm{leaf}} \nc{\twig}{\mathrm{twig}} \nc{\fe}{\mathrm{fl}} \nc{\munderline}[1]{#1} \nc{\bo}{o} \nc{\dep}{\mathrm{depth}} \nc{\ofe}{\mathrm{ofl}} \nc{\dfe}{\mathrm{dfe}} \nc{\fex}{\mathrm{fex}} \nc{\dl}{\mathrm{dlex}} \nc{\db}{\mathrm{db}} \nc{\lex}{\mathrm{lex}} \nc{\clex}{\mathrm{clex}} \nc{\dgp}{\mathrm{dgp}} \nc{\dgx}{\mathrm{dgx}} \nc{\br}{\mathrm{br}} \nc{\obd}{\mathrm{odb}} \nc{\ob}{\mathrm{ob}}
\nc{\pie}{\mathrm{PIE}}
\nc{\rbo}{\mathrm{RBO}}
\nc{\supp}{\mathcal{S}}
\nc{\nul}{\mathcal{Z}}

\nc{\bin}[2]{ (_{\stackrel{\scs{#1}}{\scs{#2}}})}  
\nc{\binc}[2]{ \left (\!\! \begin{array}{c} \scs{#1}\\
    \scs{#2} \end{array}\!\! \right )}  
\nc{\bincc}[2]{  \left ( {\scs{#1} \atop
    \vspace{-1cm}\scs{#2}} \right )}  
\nc{\bs}{\bar{S}} \nc{\cosum}{\sqsubset} \nc{\la}{\longrightarrow} \nc{\rar}{\rightarrow} \nc{\dar}{\downarrow} \nc{\dprod}{**} \nc{\dap}[1]{\downarrow \rlap{$\scriptstyle{#1}$}} \nc{\md}[1]{\bar{#1}} \nc{\uap}[1]{\uparrow \rlap{$\scriptstyle{#1}$}} \nc{\defeq}{\stackrel{\rm def}{=}} \nc{\disp}[1]{\displaystyle{#1}} \nc{\dotcup}{\ \displaystyle{\bigcup^\bullet}\ } \nc{\gzeta}{\bar{\zeta}} \nc{\hcm}{\ \hat{,}\ } \nc{\hts}{\hat{\otimes}} \nc{\barot}{{\otimes}} \nc{\free}[1]{\bar{#1}} \nc{\uni}[1]{\tilde{#1}} \nc{\hcirc}{\hat{\circ}} \nc{\leng}{\ell} \nc{\lleft}{[} \nc{\lright}{]} \nc{\lc}{\lfloor} \nc{\rc}{\rfloor}
\nc{\lb}{[} 
\nc{\rb}{]} 
\nc{\curlyl}{\left \{ \begin{array}{c} {} \\ {} \end{array}
    \right.  \!\!\!\!\!\!\!}
\nc{\curlyr}{ \!\!\!\!\!\!\!
    \left. \begin{array}{c} {} \\ {} \end{array}
    \right \} }
\nc{\longmid}{\left | \begin{array}{c} {} \\ {} \end{array}
    \right. \!\!\!\!\!\!\!}
\nc{\onetree}{\bullet} \nc{\ora}[1]{\stackrel{#1}{\rar}}
\nc{\ola}[1]{\stackrel{#1}{\la}}
\nc{\ot}{\otimes} \nc{\mot}{{{\boxtimes\,}}} \nc{\otm}{\overline{\boxtimes}} \nc{\sprod}{\bullet} \nc{\scs}[1]{\scriptstyle{#1}} \nc{\mrm}[1]{{\rm #1}} \nc{\msum}{\sum\limits}
\nc{\margin}[1]{\marginpar{\rm #1}}   
\nc{\dirlim}{\displaystyle{\lim_{\longrightarrow}}\,} \nc{\invlim}{\displaystyle{\lim_{\longleftarrow}}\,} \nc{\mvp}{\vspace{0.3cm}} \nc{\tk}{^{(k)}} \nc{\tp}{^\prime} \nc{\ttp}{^{\prime\prime}} \nc{\svp}{\vspace{2cm}} \nc{\vp}{\vspace{8cm}} \nc{\proofbegin}{\noindent{\bf Proof: }}
\nc{\proofend}{$\blacksquare$ \vspace{0.3cm}}
\nc{\modg}[1]{\!<\!\!{#1}\!\!>}
\nc{\intg}[1]{F_C(#1)} \nc{\lmodg}{\!<\!\!} \nc{\rmodg}{\!\!>\!} \nc{\cpi}{\widehat{\Pi}}
\nc{\sha}{{\mbox{\cyr X}}}  
\nc{\shap}{{\mbox{\cyrs X}}} 
\nc{\shpr}{\diamond}    
\nc{\shp}{\ast} \nc{\shplus}{\shpr^+}
\nc{\shprc}{\shpr_c}    
\nc{\msh}{\ast} \nc{\zprod}{m_0} \nc{\oprod}{m_1} \nc{\vep}{\varepsilon} \nc{\labs}{\mid\!} \nc{\rabs}{\!\mid}
\nc{\astarrow}{\overset{\raisebox{-3pt}{$\ast$}}{\rightarrow}}

\nc{\Sym}{\mrm{Sym}}
\nc{\Nsym}{\mrm{NSym}}
\nc{\Qsym}{\mrm{QSym}}
\nc{\Ssym}{\mrm{\mathfrak{S}Sym}}
\nc{\SSRCT}{\mrm{SSRCT}}
\nc{\SSRCTs}{\mrm{SSRCTs}}
\nc{\SRCT}{\mrm{SRCT}}
\nc{\SRCTs}{\mrm{SRCTs}}
\nc{\SSRT}{\mrm{SSRT}}
\nc{\SSRTs}{\mrm{SSRTs}}
\nc{\syms}{symmetric functions\xspace}
\nc{\qsyms}{quasisymmetric functions\xspace}
\nc{\nsymg}{\mathrm{NSym}_\gp}
\nc{\parr}{\mrm {Par}}
\nc{\Set}{\mrm {Set}}
\nc{\Comp}{\mrm {Comp}}
\nc{\Des}{\mrm {Des}}

\nc{\dth}{d} \nc{\mmbox}[1]{\mbox{\ #1\ }} \nc{\fp}{\mrm{FP}} \nc{\rchar}{\mrm{char}} \nc{\Fil}{\mrm{Fil}} \nc{\Mor}{Mor\xspace} \nc{\gmzvs}{gMZV\xspace} \nc{\gmzv}{gMZV\xspace} \nc{\mzv}{MZV\xspace} \nc{\mzvs}{MZVs\xspace} \nc{\Hom}{\mrm{Hom}} \nc{\id}{\mrm{id}} \nc{\im}{\mrm{im}} \nc{\incl}{\mrm{incl}}  \nc{\mchar}{\rm char}

\nc{\Alg}{\mathbf{Alg}} \nc{\Bax}{\mathbf{Bax}} \nc{\bff}{\mathbf f} \nc{\bfk}{{\bf k}} \nc{\bfone}{{\bf 1}} \nc{\bfx}{\mathbf x} \nc{\bfy}{\mathbf y}
\nc{\base}[1]{\bfone^{\otimes ({#1}+1)}} 
\nc{\Cat}{\mathbf{Cat}} \delete{}
\nc{\detail}{\marginpar{\bf More detail}
    \noindent{\bf Need more detail!}
    \svp}
\nc{\Int}{\mathbf{Int}} \nc{\Mon}{\mathbf{Mon}}
\nc{\rbtm}{{shuffle }} \nc{\rbto}{{Rota-Baxter }} \nc{\remarks}{\noindent{\bf Remarks: }} \nc{\Rings}{\mathbf{Rings}} \nc{\Sets}{\mathbf{Sets}}
\nc{\balpha}{\mathbf{\alpha}}

\nc{\BA}{{\mathbb A}} \nc{\CC}{{\mathbb C}} \nc{\DD}{{\mathbb D}} \nc{\EE}{{\mathbb E}} \nc{\FF}{{\mathbb F}} \nc{\GG}{{\mathbb G}} \nc{\HH}{{\mathbb H}} \nc{\LL}{{\mathbb L}} \nc{\NN}{{\mathbb N}} \nc{\KK}{{\mathbb K}} \nc{\PP}{{\mathbb P}} \nc{\QQ}{{\mathbb Q}} \nc{\RR}{{\mathbb R}} \nc{\TT}{{\mathbb T}} \nc{\VV}{{\mathbb V}} \nc{\ZZ}{{\mathbb Z}}

\nc{\cala}{{\mathcal A}} \nc{\calc}{{\mathcal C}} \nc{\cald}{{\mathcal D}} \nc{\cale}{{\mathcal E}} \nc{\calf}{{\mathcal F}} \nc{\calg}{{\mathcal G}} \nc{\calh}{{\mathcal H}} \nc{\cali}{{\mathcal I}} \nc{\call}{{\mathcal L}} \nc{\calm}{{\mathcal M}} \nc{\caln}{{\mathcal N}} \nc{\calo}{{\mathcal O}} \nc{\calp}{{\mathcal P}} \nc{\calr}{{\mathcal R}} \nc{\cals}{{\mathcal S}} \nc{\calt}{{\mathcal T}} \nc{\calw}{{\mathcal W}} \nc{\calk}{{\mathcal K}} \nc{\calx}{{\mathcal X}}
\nc{\calz}{{\mathcal Z}}
 \nc{\CA}{\mathcal{A}}

\nc{\fraka}{{\mathfrak a}} \nc{\frakA}{{\mathfrak A}} \nc{\frakb}{{\mathfrak b}} \nc{\frakB}{{\mathfrak B}}
\nc{\frakc}{{\mathfrak c}}  \nc{\frakD}{{\mathfrak D}}
\nc{\frakH}{{\mathfrak H}}
\nc{\frakh}{{\mathfrak h}} \nc{\frakM}{{\mathfrak M}}
\nc{\frakO}{{\mathfrak O}}
\nc{\frakE}{{\mathfrak E}}
\nc{\bfrakM}{\overline{\frakM}} \nc{\frakm}{{\mathfrak m}} \nc{\frakP}{{\mathfrak P}} \nc{\frakN}{{\mathfrak N}} \nc{\frakp}{{\mathfrak p}} \nc{\frakS}{{\mathfrak S}}
\nc{\frakk}{{\mathfrak k}}
\nc{\frakx}{{\mathfrak x}}
\nc{\frakl}{{\mathfrak l}} \nc{\ox}{\bar{\frakx}} \nc{\frakX}{{\mathfrak X}} \nc{\fraky}{{\mathfrak y}} \nc\dop{\delta}
\nc{\Reduce}{{\rm Red}}

\font\cyr=wncyr10 \font\cyrs=wncyr7
\nc{\redt}[1]{\textcolor{red}{#1}}
\nc{\li}[1]{\textcolor{red}{\tt Li:#1}}
\nc{\huoyi}[1]{\textcolor{blue}{\tt sz:#1}}

\begin{document}
\title{Rigidity for the Hopf algebra of quasisymmetric functions}
\author{Wanwan Jia}
\address{School of Mathematics and Statistics, Southwest University, Chongqing 400715, China}
\email{jiawanwan919@163.com}
\author{Zhengpan Wang}
\address{School of Mathematics and Statistics, Southwest University, Chongqing 400715, China}
\email{zpwang@swu.edu.cn}
\author{Houyi Yu\textsuperscript{*}}\thanks{*Corresponding author}
\address{School of Mathematics and Statistics, Southwest University, Chongqing 400715, China}
\email{yuhouyi@swu.edu.cn}

\hyphenpenalty=8000
\date{}

\begin{abstract}
We investigate the rigidity for the Hopf algebra ${\rm QSym}$ of quasisymmetric functions with respect to the  monomial, the fundamental and the quasisymmetric Schur basis, respectively.
By establishing some combinatorial properties of the posets of  compositions arising from the analogous Pieri rules for quasisymmetric functions,
we show that ${\rm QSym}$ is rigid as an algebra with respect to the quasisymmetric Schur basis,
and rigid as a coalgebra with respect to the monomial and the quasisymmetric Schur basis, respectively.
The natural actions of reversal, complement and transpose of the labelling compositions lead to some nontrivial graded (co)algebra automorphisms of ${\rm QSym}$.
We prove that the linear maps induced by the three actions are precisely the only nontrivial
graded algebra automorphisms that take the fundamental basis into itself. Furthermore,
the complement map on the labels gives the unique nontrivial graded coalgebra automorphism preserving the fundamental basis, while the
reversal map on the labels gives the unique nontrivial graded algebra automorphism preserving the monomial basis.
Therefore, ${\rm QSym}$ is rigid as a Hopf algebra with respect to the monomial  and the quasisymmetric Schur basis.
\end{abstract}

\keywords{rigidity, quasisymmetric function, Hopf algebra,  automorphism}

\maketitle

\hyphenpenalty=8000 \setcounter{section}{0}


\allowdisplaybreaks

\section{Introduction}\label{sec:int}

Reutenauer and his school \cite{MR95,PR95} introduced the Hopf algebra $\Ssym$ of permutations, which has a linear basis indexed by permutations in all symmetric groups $\mathfrak{S}_n$.
One of its important quotient Hopf algebras is the Hopf algebra of quasisymmetric functions, ${\rm {\Qsym}}$,
which was first introduced by Gessel
\cite{Ge84} in the 1980s to deal with the combinatorics of Stanley's P-partitions \cite{Sta2},
and since then ${\rm {\Qsym}}$ has been extensively studied in
algebraic combinatorics \cite{ABS2006,BR2008,BMSW2000,GKLLRT1995,Hi2000,Kw2009}.
Much of the combinatorial richness arising from these Hopf algebras stems from their various distinguished bases and the relationships among these bases.
A graded algebra (respectively, bialgebra, Hopf algebra) is said to be rigid with respect to a given graded basis if there are no nontrivial algebra
(respectively, bialgebra, Hopf algebra) automorphisms that
take the given basis into itself as a graded set.
Hazewinkel \cite{Ha07} proved that the algebra $\Ssym$ is rigid as a Hopf algebra with distinguished basis, i.e., there are no
nontrivial automorphisms that preserve the multiplication and comultiplication and take the distinguished basis of all permutations into itself as a graded set.
Based on this result, Hazewinkel conjectured that there are similar results for the Hopf algebra of quasisymmetric functions.

Motivated by the conjecture of Hazewinkel, the major goal of this paper is to understand the rigidity for the Hopf algebra ${\rm {\Qsym}}$.
As is well known, ${\rm {\Qsym}}$ has various bases \cite{BBSS15,BJR09,Ge84,HLMW11,Lu08,LMW13,MR14,TW15}, all indexed by compositions. In this paper, we focus on the monomial basis $M_\alpha$, the fundamental basis $F_\alpha$
and the quasisymmetric Schur basis ${\mathcal{S}}_{\alpha}$.
The notions of complement, reversal and transpose of compositions correspond to
 well-known involutive graded algebra automorphisms of ${\rm {\Qsym}}$ defined in terms of the
fundamental basis as follows \cite{LMW13}:
\begin{align}
\Psi:{\rm {\Qsym}}\rightarrow{\rm {\Qsym}},\qquad \Psi(F_{\alpha})=F_{\alpha^c};\nonumber\\
\rho:{\rm {\Qsym}}\rightarrow{\rm {\Qsym}},\qquad \rho(F_{\alpha})=F_{\alpha^r};\label{iafcom}\\
\omega:{\rm {\Qsym}}\rightarrow{\rm {\Qsym}},\qquad \omega(F_{\alpha})=F_{\alpha^t}.\nonumber
\end{align}
Note that these automorphisms commute with each other, and that $\omega=\Psi\circ\rho=\rho\circ\Psi$.
Malvenuto and Reutenauer proved in \cite[Th\'eor\`eme 4.12 and page 50]{M93} and \cite[Corollary 2.4]{MR95} that $\rho$ and $\omega$ are algebra automorphisms of ${\rm {\Qsym}}$
and so does  $\Psi$, since  $\Psi=\rho\circ\omega$. (As they mentioned Gessel had shown the results in  \cite{Ge90}).
We will show that $\Psi$ is also the only nontrivial comultiplication-preserving map so that it is the unique nontrivial graded
Hopf algebra automorphism of ${\rm {\Qsym}}$ that preserves the fundamental basis, and that $\rho$ is the unique nontrivial graded algebra automorphism of ${\rm {\Qsym}}$ that takes the monomial basis into itself, but it does not preserve the comultiplication.
The set of quasisymmetric Schur functions is a basis of ${\rm {\Qsym}}$ discovered by Haglund,  Luoto,  Mason and Willigenburg \cite{HLMW11}. We will show that ${\rm {\Qsym}}$
is rigid as a Hopf algebra with respect to this kind of basis, that is, there are no
nontrivial automorphisms that are multiplication-preserving or comultiplication-preserving and
preserve the quasisymmetric Schur basis.

More precisely, this paper is structured as follows.
Section \ref{sec:background} briefly gives background on quasisymmetric functions required to state and prove our results.
In order to prove the main results, we review in Section \ref{sec:perposet3} several posets of compositions arising from the analogous Pieri rules for quasisymmetric functions
and investigate some of their combinatorial properties.
Then in Sections \ref{sec:rwrttmb} we show that the map $\rho$ defined by Eq.\eqref{iafcom} is the only nontrivial graded algebra automorphism of
${\rm {\Qsym}}$ that preserves the set of monomial quasisymmetric functions.
But ${\rm {\Qsym}}$ is rigid as a coalgebra with respect to the monomial basis so that it is also rigid as a Hopf algebra.
We focus our attention on the fundamental basis in Section \ref{sec:rwrttfb}, showing that these maps $\Psi,\rho,\omega$  defined by Eq.\eqref{iafcom}
are the only nontrivial graded algebra automorphisms of ${\rm {\Qsym}}$ that preserve the fundamental basis, and that the map $\Psi$ is the unique nontrivial
graded coalgebra automorphism of ${\rm {\Qsym}}$ preserving the fundamental basis. So $\Psi$ is the only nontrivial
graded Hopf algebra automorphism preserving the fundamental basis.
Section \ref{sec:rwrttqsf}  is devoted to showing that ${\rm {\Qsym}}$ is rigid as a graded  Hopf algebra with respect to the quasisymmetric Schur basis.


\section{Background and Preliminaries}\label{sec:background}

In this section, we review  background on quasisymmetric functions that will be useful to us later.
We refer the reader to~\cite{Sta,Sta2}, as well as~\cite{Ge84,HLMW11,LMW13,MR95}, for further details on quasisymmetric functions.

A {\em composition} $\alpha=(\alpha_1,\alpha_2,\cdots,\alpha_k)$ of a positive integer $n$, often denoted by $\alpha\models n$, is a finite ordered list of positive integers whose sum is $n$. The set of all compositions will be denoted by $\mathcal{C}$.
We shall often omit the commas and write a composition $(\alpha_1,\alpha_2,\cdots,\alpha_k)$ as $\alpha_1\alpha_2\cdots\alpha_k$,
when it is used as a subscript and no confusion can arise.
Given a composition $\alpha=(\alpha_1,\alpha_2,\cdots,\alpha_k)$ we call the integers $\alpha_i$ the {\em parts} of $\alpha$
and define its {\em weight} to be $|\alpha|=\alpha_1+\alpha_2+\cdots+\alpha_k$ and its {\em length} to be
$\ell(\alpha)=k$.
If $\alpha_j =\cdots=\alpha_{j+m-1}=a$ we often
abbreviate this sublist to $a^m$.
A {\em partition} is a composition whose parts are weakly decreasing.
The {\em underlying partition} of a composition $\alpha$, denoted by $\widetilde{\alpha}$, is the partition obtained by sorting the parts of $\alpha$ into weakly decreasing order.
For convenience we denote by $\emptyset$ the unique composition (respectively, partition) of weight and length $0$,
called the {\em empty composition} (respectively, {\em empty partition}).

For positive integers $m,n$ with $m\leq n$, we define the {\em interval} $[m,n]=\{m,m+1,\cdots,n\}$. In particular, we write the interval $[1,n]$ as $[n]$
for short, and hence $[n]=\{1,2,\cdots,n\}$. Any composition $\alpha=(\alpha_1,\alpha_2,\cdots,\alpha_k)$ of $n$ naturally corresponds
to a subset $\Set(\alpha)$ of $[n-1]$
where
\begin{align*}
\Set(\alpha)=\{\alpha_1,\alpha_1+\alpha_2,\cdots,\alpha_1+\alpha_2+\cdots+\alpha_{k-1}\}.
\end{align*}
Note that this is an invertible procedure; that is, any subset $S=\{s_1,s_2\cdots,s_k\}\subseteq [n-1]$ with $s_1<s_2<\cdots<s_k$ corresponds to a composition $\Comp(S)\models n$
where
\begin{align*}
{\rm Comp} (S)=(s_1,s_2-s_1,\cdots,s_{k}-s_{k-1},n-s_{k}).
\end{align*}
In particular, the empty set corresponds to the composition $\emptyset$ if $n=0$, and to $(n)$ if $n>0$.

Given a composition $\alpha=(\alpha_1,\alpha_2,\cdots,\alpha_k)$ there exist three
 closely related compositions: its reversal, its complement, and its transpose.
 Firstly, the {\em reversal} of $\alpha$, denoted by $\alpha^r$, is obtained by writing the parts of
$\alpha$ in the reverse order. Secondly, the {\em complement} of $\alpha$, denoted by $\alpha^c$, is given by
$
\alpha^c=\Comp\left(\Set(\alpha)^c\right),
$
that is, $\Set(\alpha^c)=[|\alpha|-1]-\Set(\alpha)$,
and hence $\alpha^c$ is the composition that corresponds to the complement of the set that corresponds to $\alpha$.
Finally, the {\em transpose} (also known as the {\em conjugate}) of $\alpha$, denoted by $\alpha^t$, is defined to be
$\alpha^t=\left(\alpha^r\right)^c$, which is also equal to $\left(\alpha^c\right)^r$.

Given compositions $\alpha,\beta$, we say that $\alpha$ is {\em a refinement} of $\beta$, denoted by $\alpha\preceq\beta$,  if $\beta$ is obtained from
$\alpha$ by summing some consecutive parts of $\alpha$. Thus, $\alpha\preceq \beta$ if and only if $|\alpha|=|\beta|$ and $\Set(\beta)\subseteq\Set(\alpha)$.
The {\em concatenation} of
$\alpha=(\alpha_1,\alpha_2,\cdots,\alpha_k)$
and $\beta=(\beta_1,\beta_2,\cdots,\beta_l)$ is
\begin{align*}
\alpha\cdot \beta=(\alpha_1,\cdots,\alpha_k,\beta_1,\cdots,\beta_l)
\end{align*}
while the {\em near concatenation} is
\begin{align*}
\alpha\odot \beta=(\alpha_1,\cdots,\alpha_k+\beta_1,\cdots,\beta_l).
\end{align*}

A {\em quasisymmetric function} is a bounded degree formal power series $f\in \QQ[[x_1,x_2,\cdots]]$ such that for each
composition $(\alpha_1,\alpha_2,\cdots,\alpha_k)$, all monomials $x_{i_1}^{\alpha_1}x_{i_2}^{\alpha_2}\cdots x_{i_k}^{\alpha_k}$ in $f$ with
indices $i_1<i_2<\cdots<i_k$ have the same coefficient.
Let ${\rm {\Qsym}}$ denote the set of all quasisymmetric functions. Then ${\rm {\Qsym}}$ forms a connected graded Hopf algebra
$${\rm {\Qsym}}=\bigoplus_{n\geq0}{\rm {\Qsym}}_n,$$
where ${\rm {\Qsym}}_n$ is the space of homogeneous quasisymmetric functions of degree $n$.

There are a number of well-known bases for ${\rm {\Qsym}}$, all indexed by compositions. For the following three
considered here, see \cite{Ge84} and \cite{HLMW11}. The {\em monomial basis}
consists of $M_{\emptyset}=1$ and all formal power series
\begin{align*}
M_\alpha=\sum_{i_1<i_2<\cdots<i_k}x_1^{\alpha_1}x_2^{\alpha_2}\cdots x_k^{\alpha_k},
\end{align*}
where  $\alpha=(\alpha_1,\alpha_2,\cdots,\alpha_k)$,
while the {\em fundamental basis} consists of $F_\emptyset=1$ and all formal power series
\begin{align*}
F_\alpha=\sum_{\beta\preceq\alpha}M_{\beta}.
\end{align*}
Hence by the M\"obius inversion formula,
\begin{align}\label{eq:MFbas}
M_\alpha=\sum_{\beta\preceq\alpha}(-1)^{\ell(\beta)-\ell(\alpha)}F_{\beta}.
\end{align}
Moreover, ${\rm {\Qsym}}_n={\rm span}\{M_\alpha|\alpha\models n\}={\rm span}\{F_\alpha|\alpha\models n\}$.

As an algebra ${\rm {\Qsym}}$ is just a subalgebra of the formal power series algebra $\QQ[[x_1,x_2,\cdots]]$ in countable many commuting variables.
So the product of two quasisymmetric functions is the natural multiplication of formal power series.
The coproduct on each of these bases is even more straightforward to describe
\begin{align}\label{qsprocop}
\Delta(M_\alpha)=\sum_{\alpha=\beta\cdot\gamma}M_{\beta}\otimes M_{\gamma}, \qquad
\Delta(F_\alpha)=\sum_{\alpha=\beta\cdot\gamma}F_{\beta}\otimes F_{\gamma}+\sum_{\alpha=\beta\odot\gamma}F_{\beta}\otimes F_{\gamma}.
\end{align}
The counit is given by
\begin{align}\label{counitqs}
\varepsilon(M_\alpha)=
\begin{cases}
1 &if\ \alpha=\emptyset\\
0 &otherwise,
\end{cases}
\qquad
\varepsilon(F_\alpha)=
\begin{cases}
1 &if \ \alpha=\emptyset\\
0 &otherwise.
\end{cases}
\end{align}

The basis of quasisymmetric Schur functions \cite{HLMW11} is introduced for the Hopf algebra ${\rm {\Qsym}}$ in 2011.
It arose from the combinatorics of Macdonald polynomials \cite{HHL2005} and itself initiated a search for
other Schur-like bases of ${\rm {\Qsym}}$ such as row-strict quasisymmetric functions \cite{MR14} and dual
immaculate quasisymmetric functions \cite{BBSS15}.
We next briefly recall the concept of quasisymmetric Schur functions, see \cite{BLW2011,HLMW11,LMW13} for more details.

We depict a composition $\alpha=(\alpha_1,\alpha_2,\cdots,\alpha_{k})\models n$ by using its {\em reverse composition diagram},
also denoted by $\alpha$, which is the left-justified array of $n$ cells with $\alpha_i$ cells in the $i$-th row, where, following the English convention, we
number the rows from top to bottom, and the columns from left to right. The cell in the $i$-th row and $j$-th column is denoted by the pair $(i,j)$.
For example,
\vspace{3mm}
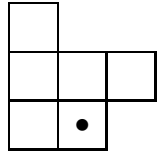
\begin{center}
\begin{tabular}{c}
\begin{ytableau}
\, \\
\,&\,&\,\\
\,&\bullet\\
\end{ytableau}
\end{tabular}
\end{center}
\vspace{3mm}
is the reverse composition diagram of $(1,3,2)$, and the cell filled with a $\bullet$ is the cell $(3,2)$.

We call the reverse composition diagram of $\alpha$ the {\em Young diagram} if $\alpha$  is a partition.
Let $\alpha,\beta$ be two Young diagrams. We say that $\beta$ is contained in $\alpha$, denoted by $\beta\subseteq\alpha$, if $\ell(\beta)\leq\ell(\alpha)$
and $\beta_i\leq\alpha_i$ for $1\leq i\leq \ell(\beta)$. If $\beta\subseteq\alpha$, then the {\em skew partition shape} $\alpha/\beta$ is the array of cells
\begin{align*}
\alpha/\beta=\{(i,j)|(i,j)\in \alpha\ {\rm and}\ (i,j)\not\in \beta\},
\end{align*}
where $\beta$ is positioned in the upper left corner of $\alpha$.  The {\em size} of $\alpha/\beta$ is $|\alpha/\beta|=|\alpha|-|\beta|$.

We must distinguish between the skew partition shape $\alpha/\beta$ and the skew reverse composition shape
$\alpha/\hspace{-1.2mm}/\beta$, which is defined below.
\begin{definition}\label{reversecompposet}
(Definition 4.2.3 in \cite{LMW13}) The {\em reverse composition poset} $\mathcal{L}_{C}$ is the poset consisting of all compositions in which
$\alpha=(\alpha_1,\alpha_2,\cdots,\alpha_k)$ is covered by $\beta$, denoted by $\alpha\prec_{C}\beta$, if either
\begin{enumerate}
\item $\beta=(1,\alpha_1,\cdots,\alpha_{k})$, that is, the composition obtained by prepending $\alpha$ with a new part of size $1$, or

\item $\beta=(\alpha_1,\cdots,\alpha_j+1,\cdots,\alpha_{k})$, provided that $\alpha_i\neq \alpha_j$ for all $i$ with $1\leq i<j$, that is, the composition obtained
by adding $1$ to a part of $\alpha$ as long as that part is the leftmost part of that size.
\end{enumerate}
The partial order on $\mathcal{L}_{C}$ is denoted by $\leq_{C}$. For example, a saturated chain in $\mathcal{L}_{C}$ is $(1)\prec_{C}(1,1)\prec_{C}(2,1)\prec_{C}(1,2,1)$.
\end{definition}

Let $\alpha,\beta$ be two reverse composition diagrams such that $\beta\leq_{C}\alpha$.
Then we define the {\em skew reverse composition shape}
$\alpha/\hspace{-1.2mm}/\beta$ to be the array of cells
\begin{align*}
\alpha/\hspace{-1.2mm}/\beta=\{(i,j)|(i,j)\in \alpha\ {\rm and}\ (i,j)\not\in \beta\},
\end{align*}
where $\beta$ has been drawn in the bottom left corner. The {\em size} of $\alpha/\hspace{-1.2mm}/\beta$ is
$|\alpha/\hspace{-1.2mm}/\beta|=|\alpha|-|\beta|$.

Following the vocabulary regarding skew shapes (either partition or composition), we call $\beta$ the {\em inner shape}
and $\alpha$ the {\em outer shape}. A  skew shape is a {\em horizontal strip} if no two cells lie in the same
column, and a {\em vertical strip} if no two cells lie in the same row.

\begin{example}
The skew shapes $(5,4,2,2)/(4,2,2)$ and $(1,2,4,2)/\hspace{-1.2mm}/(1,3,1)$ are shown below, respectively, where the inner shapes are denoted by cells filled with a $\bullet$.
\vspace{3mm}
\begin{center}
\begin{tabular}{c}
\begin{ytableau}
\bullet&\bullet&\bullet\ &\bullet\ &  \\
\bullet&\bullet&\, &\, \\
\bullet&\bullet  \\
\,     &\,      \\
\end{ytableau}\\
\\
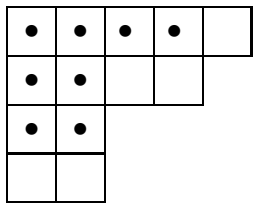
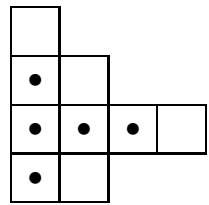
Skew shape $(5,4,2,2)/(4,2,2)$
\end{tabular}
\qquad
\begin{tabular}{c}
\begin{ytableau}
   \\
\bullet&  \\
\bullet&\bullet&\bullet & \\
\bullet& \\
\end{ytableau}\\
\\
Skew shape $(1,2,4,2)/\hspace{-1.2mm}/(1,3,1)$
\end{tabular}
\end{center}
Note that the first one is a horizontal strip, while the second one is a vertical strip.
\end{example}

\begin{definition}\label{defnSSRCT} (Definition 4.2.6 in \cite{LMW13})
Given a skew reverse composition of shape $\alpha/\hspace{-1.2mm}/\beta$, we define a \emph{semistandard reverse composition tableau} (abbreviated to $\SSRCT$) ${\tau}$ of shape $\alpha/\hspace{-1.2mm}/\beta$ to be a filling
\begin{align*}
{\tau}: \alpha/\hspace{-1.2mm}/\beta\rightarrow \mathbb{Z}^{+}
\end{align*}
of the cells $(i,j)$ of $\alpha/\hspace{-1.2mm}/\beta$
such that
\begin{enumerate}
\item the entries in each row are weakly decreasing when read from left to right,
\item the entries in the first column are strictly increasing when read from top to bottom,
\item set ${\tau}(i,j)=\infty$ for all $(i,j)\in \beta$, and if $i<j$, $(j,k+1)\in \alpha/\hspace{-1.2mm}/\beta$
and ${\tau}(i,k)\geq {\tau}(j,k+1)$, then $(i,k+1)\in\alpha/\hspace{-1.2mm}/\beta$ and ${\tau}(i,k+1)>{\tau}(j,k+1)$.
\end{enumerate}
Sometimes we will abuse notation and use $\SSRCT$ to denote the set of all such tableaux.
\end{definition}

Given an $\SSRCT$ ${\tau}$, we define the \emph{content} of ${\tau}$, denoted by cont$({\tau})$, to be the list of nonnegative integers
\begin{align*}
{\rm cont}({\tau})=(c_1,c_2,\cdots,c_{max}),
\end{align*}
where $c_i$ is the number of times $i$ appearing in ${\tau}$, and $max$ is the largest integer appearing in ${\tau}$. Given variables
$x_1,x_2,\cdots$, we define the \emph{monomial} of ${\tau}$ to be
\begin{align*}
x^{{\tau}}={x_1}^{c_1}{x_2}^{c_2}\cdots{x_{max}}^{c_{max}}.
\end{align*}

\begin{example}\label{exam:(S)SRCT}
An SSRCT of the shape $(3,4,2,3)/\hspace{-1.2mm}/(1,2)$ is shown below, where the inner shape is denoted by cells filled with a $\bullet$.
\vspace{3mm}
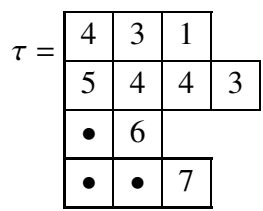
\begin{center}
\begin{tabular}{c}
$\tau=$\ \begin{ytableau}
$4$&$3$&$1$   \\
$5$&$4$&$4$&$3$ \\
\bullet&$6$  \\
\bullet&\bullet&$7$ \\
\end{ytableau}
\end{tabular}
\qquad
\begin{tabular}{c}
${\rm cont}(\tau)=(1,0,2,3,1,1,1)$\\
\, \ \ \ \ \ \ \  $x^{\tau}={x_1}{x_3}^{2}{x_4}^{3}{x_5}{x_6}{x_7}$
\end{tabular}
\end{center}
\end{example}
\vspace{3mm}

Let $\alpha$ be a composition. Then the {\emph{quasisymmetric Schur functions}} \cite[Definition 5.1]{HLMW11} ${\mathcal{S}}_{\alpha}$ is defined by
${\mathcal{S}}_{\emptyset}=1$ and
$$
{\mathcal{S}}_{\alpha}=\sum_{{\tau}}x^{{\tau}}
$$
where the sum is over all  $\SSRCTs$ ${\tau}$ of shape $\alpha$.
By \cite[Proposition 5.5]{HLMW11}, the set of all quasisymmetric Schur functions forms a basis for ${\rm QSym}$.
Let $\gamma$ be a composition. According to  \cite[Theorem 3.5]{BLW2011}, the coproduct in terms of the quasisymmetric Schur basis is defined by
\begin{align}\label{eq:deltasagammacoprod}
\Delta({\mathcal{S}}_{\gamma})=\sum_{\alpha,\beta} {C}_{\alpha\beta}^{\gamma}{\mathcal{S}}_{\alpha}\otimes {\mathcal{S}}_{\beta}
\end{align}
where the sum is over all compositions $\alpha,\beta$ and ${C}_{\alpha\beta}^{\gamma}$ is the \emph{ noncommutative Littlewood-Richardson coefficients}.
The values of some special noncommutative Littlewood-Richardson coefficients, that will be used later, can be given directly by noncommutative Pieri rules \cite[Corollary 3.8]{BLW2011}.
More precisely,  for any positive integer $n$, we have
\begin{align}\label{eq:nlrcnvsp}
C_{1^n,\beta}^{\gamma}=
\begin{cases}
1 &if\ {\beta\leq_C\gamma\ {such\ that}\ |\gamma/\hspace{-1.2mm}/\beta|=n} \ {{and}\ \gamma/\hspace{-1.2mm}/\beta\ {is\ a\ vertical\ strip}},\\
0 &otherwise.
\end{cases}
\end{align}
In particular, $C_{1,\beta}^{\gamma}=1$ if and only if $\beta\prec_{C}\gamma$.

For convenience of notation, for each $K\in\{M,F,\mathcal{S}\}$ we define a scalar product on ${\rm {\Qsym}}$, i.e., a $\mathbb{Q}$-valued bilinear form $\langle \cdot,\cdot \rangle_K$, by requiring that $\langle K_{\alpha},K_{\beta} \rangle_K=\delta_{\alpha\beta}$ for all compositions $\alpha$, $\beta$, where $\delta_{\alpha\beta}$ is the Kronecker delta.
Then we have
\begin{align}\label{eq:KKKprod}
K_\alpha K_\beta=\sum_{\gamma\in\mathcal{C}}\langle K_{\gamma},K_\alpha K_\beta\rangle_K  K_\gamma.
\end{align}

We remark that a connected graded bialgebra $H$ has a unique antipode $S$, endowing it with a Hopf structure \cite{Ehr,GR2018}.
Thus, a bialgebra endomorphism of ${\rm {\Qsym}}$ is automatically a Hopf endomorphism.
For further references relevant to Hopf algebras, we refer the reader to \cite{Swe1969}.


\section{Partial orders on the set of compositions}\label{sec:perposet3}

In Section \ref{sec:background}, we recalled the reverse composition poset $(\mathcal{L}_{C},\leq_{C})$ in order to define the quasisymmetric Schur functions.
In this section we will review three other partial orders on the set $\mathcal{C}$ of all compositions, and deduce some of their order-theoretic properties,
which will help us to state Pieri-type rules for quasisymmetric functions.

The multiplication of two monomial quasisymmetric functions is described  in terms of the {\em quasi-shuffle
product} $*$ \cite{Ehr,Ho} on the indexing compositions.
Let
$\alpha=(\alpha_1,\alpha_2,\cdots,\alpha_k)$  and $\beta=(\beta_1,\beta_2,\cdots,\beta_l)$
be two compositions. The quasi-shuffle $*$
can be recursively defined by
\begin{equation*}
\alpha *  \beta =(\alpha_1,(\alpha_2, \cdots, \alpha_k) *  \beta)+(\beta_1,\alpha *  (\beta_2, \cdots,\beta_l))
+(\alpha_1+\beta_1,(\alpha_2,\cdots,a_k)*   (\beta_2,\cdots,\beta_l)),
\end{equation*}
with the convention that $\emptyset* \alpha=\alpha* \emptyset=\alpha$.
We sometimes write the quasi-shuffle product as
$$
\alpha*\beta=\sum_{\delta\in\mathcal{C}}\langle\delta,\alpha*\beta\rangle\delta,
$$
where $\langle\delta,\alpha*\beta\rangle$ is the coefficient in  $\alpha*\beta$ of the composition $\delta$.
Consequently, by Eq.\eqref{eq:KKKprod}, we have $\langle M_{\delta},M_{\alpha}M_{\beta}\rangle_{M}=\langle\delta,\alpha*\beta\rangle$
for all compositions $\alpha$, $\beta$, $\delta$, and hence
\begin{align}\label{quasishuprodqsm}
M_{\alpha} M_{\beta}=\sum_{\delta\in\mathcal{C}}\langle\delta,\alpha*\beta\rangle M_{\delta}.
\end{align}
For example,
\begin{align*}
M_{132}M_{2}=&2M_{1322}+M_{1232}+M_{2132}+M_{134}+M_{152}+M_{332}.
\end{align*}

The partial order relating to the Pieri-type rules for monomial quasisymmetric functions was introduced by Bergeron, Bousquet-M\'elou and Dulucq in \cite{BBD95}.
A composition $\alpha$ is said to be covered by a composition $\beta$, denoted by $\alpha\prec_M\beta$,  if $\beta$ is obtained either by adding $1$ to a part of $\alpha$,
or by inserting a part of size $1$ anywhere into $\alpha$. The partial order obtained by transitive closure of this covering relation
is denoted $\leq_M$ and the poset thus obtained is denoted by $(\Gamma,\leq_M)$.
For any compositions $\alpha,\beta$, it follows from Eq.\eqref{quasishuprodqsm} that  the product $M_{\alpha}M_{\beta}$ of two monomial quasisymmetric functions is {\emph{M-positive}}, that is, it is a nonnegative linear combination of monomial quasisymmetric functions.
In particular, we have the following Pieri-type rules for monomial quasisymmetric functions.
\begin{lemma}\label{lem:m1dalmonomialb}
Let $\alpha=(\alpha_1,\alpha_2,\cdots,\alpha_k)$ be a composition. Then
\begin{align}\label{monomialqsymbasm1malp}
M_{1}M_{\alpha}=\sum_{r=0}^{k}M_{(\alpha_1,\cdots,\alpha_r,1,\alpha_{r+1},\cdots,\alpha_k)}
+\sum_{r=1}^{k}M_{(\alpha_1,\cdots,\alpha_{r-1},\alpha_r+1,\alpha_{r+1},\cdots,\alpha_k)}.
\end{align}
In particular, $\langle M_{\beta},M_{1}M_{\alpha}\rangle_M>0$ if and only if $\alpha\prec_M\beta$.
\end{lemma}
The proof is a straightforward argument using Eq.\eqref{quasishuprodqsm} and hence will be omitted.

In order to establish the Pieri-type rules for fundamental quasisymmetric functions, Bj\"orner and Stanley \cite{BS05} introduced
the following partial order on $\mathcal{C}$. We say that a composition
$\beta$ covers $\alpha=(\alpha_1, \cdots , \alpha_k)$, denoted by $\alpha\prec_F\beta$, if $\beta$ can be obtained from $\alpha$
either by adding $1$ to a part, or adding $1$ to a part and then splitting
this part into two parts. More precisely, for some $j$ we have either
$$\beta=(\alpha_1,\cdots, \alpha_{j-1}, \alpha_{j}+1,\alpha_{j+1}, \cdots, \alpha_k)$$
or
$$\beta=(\alpha_1,\cdots, \alpha_{j-1},h, \alpha_{j}+1-h, \alpha_{j+1}, \cdots, \alpha_k)$$
for some $1\leq h\leq \alpha_j$. The partial order obtained by transitive closure of this covering relation
is denoted by $\leq_F$ and the poset thus obtained is denoted by $(\mathcal{F},\leq_F)$. In \cite[Section 3]{BS05}, the authors proved the following Pieri-type rules for fundamental quasisymmetric functions:
\begin{align}\label{fundbasf1falp}
F_1F_{\alpha}=\sum_{\alpha\prec_F \beta}F_{\beta}.
\end{align}

In order to state the analogous Pieri rules for  quasisymmetric Schur functions \cite[Theorem 6.3]{HLMW11},
we need three operators: $\text{rem}$, $\text{row}$, $\text{col}$.
For a given composition $\alpha$, denote
$\text{rem}_s(\alpha)$ the composition obtained by subtracting $1$ from the rightmost part of size $s$ in $\alpha$, and removing
any $0$ that might arise in this process. If there is no such part then we define $\text{rem}_s(\alpha)=\emptyset$.
Given positive integers $s_1<s_2<\cdots<s_j$ and $m_1\leq m_2\leq \cdots\leq m_j$, we define
\begin{align*}
\text{row}_{\{s_1,s_2,\cdots,s_j\}}(\alpha)=\text{rem}_{s_1}(\cdots(\text{rem}_{s_{j-1}}(\text{rem}_{s_j}(\alpha)))\cdots)
\end{align*}
and
\begin{align*}
\text{col}_{\{m_1,m_2,\cdots,m_j\}}(\alpha)=\text{rem}_{m_j}(\cdots(\text{rem}_{m_{2}}(\text{rem}_{m_1}(\alpha)))\cdots).
\end{align*}
For example, if $\alpha=(1,2,3)$, then
\begin{align*}
\text{row}_{\{2,3\}}(\alpha)=\text{rem}_{2}(\text{rem}_{3}(\alpha))=\text{rem}_{2}(1,2,2)=(1,2,1)
\end{align*}
and
\begin{align*}
\text{col}_{\{2,3\}}(\alpha)=\text{rem}_{3}(\text{rem}_{2}(\alpha))=\text{rem}_{3}(1,1,3)=(1,1,2).
\end{align*}

For any horizontal strip $\delta$ we denote by $S(\delta)$ the set of columns its skew diagram occupies, and for any vertical strip $\epsilon$
we denote by $M(\epsilon)$ the multiset of columns its skew diagram occupies, where multiplicities for a column are given by the number of cells in that column,
and column indices are listed in weakly increasing order.

From \cite[Theorem 6.3]{HLMW11}, we have the following analogous Pieri rules for quasisymmetric Schur functions.

\begin{lemma}[Pieri rules for quasisymmetric Schur functions]\label{PieriruqSchurf}
Let $\alpha$ be a composition. Then
\begin{enumerate}
\item\label{eq:itemPieriruleqsfsnsal}
$$
{\mathcal{S}}_{n}{\mathcal{S}}_{\alpha}=\sum_{\beta}{\mathcal{S}}_{\beta}
$$
where the sum is taken over all compositions $\beta$ such that
\begin{enumerate}
\item\label{lem:item1Pieriqsyms} $\delta=\widetilde{\beta}/\widetilde{\alpha}$ is a horizontal strip,
\item\label{lem:item2Pieriqsyms} $|\delta|=n$,
\item\label{lem:item3Pieriqsyms} ${\rm row}_{S(\delta)}(\beta)=\alpha$;
\end{enumerate}
\item\label{eq:itemPieriruleqsfs1nsal}
$$
{\mathcal{S}}_{1^n}{\mathcal{S}}_{\alpha}=\sum_{\beta}{\mathcal{S}}_{\beta}
$$
where the sum is taken over all compositions $\beta$ such that
\begin{enumerate}
\item $\epsilon=\widetilde{\beta}/\widetilde{\alpha}$ is a vertical strip,
\item $|\epsilon|=n$,
\item ${\rm col}_{M(\epsilon)}(\beta)=\alpha$.
\end{enumerate}
\end{enumerate}
\end{lemma}

\begin{example}\label{lem:s11s11proguctss2s2}
According to Lemma \ref{PieriruqSchurf}, we have
$$
{\mathcal{S}}_{11}{\mathcal{S}}_{11}={\mathcal{S}}_{22}+{\mathcal{S}}_{211}+{\mathcal{S}}_{121}+{\mathcal{S}}_{112}+{\mathcal{S}}_{1111},\quad
{\mathcal{S}}_{2}{\mathcal{S}}_{2}={\mathcal{S}}_{4}+{\mathcal{S}}_{31}+{\mathcal{S}}_{22}+{\mathcal{S}}_{13}.
$$
\end{example}

The Pieri rules for symmetric functions \cite[Theorem 7.15.7]{Sta2} give rise to Young's lattice on partitions in the following way.
Let $\lambda,\mu$ be partitions. Then $\lambda$ covers $\mu$ in Young's lattice if the coefficient of $s_{\lambda}$ in $s_{1}s_{\mu}$ is $1$, where
$s_{\lambda}$ and $s_{\mu}$ are the Schur functions. Analogously,
Lemma \ref{PieriruqSchurf} gives rise to a poset, denoted by $(\mathcal{Q}_{C},\leq_Q)$,  on compositions:
if $\alpha$ and $\beta$ are compositions, then $\beta$ covers $\alpha$ in the poset $\mathcal{Q}_{C}$ given that the coefficient of ${\mathcal{S}}_{\beta}$ in
${\mathcal{S}}_{1}{\mathcal{S}}_{\alpha}$ is $1$. In other words,
we have
\begin{align}\label{qsymschurbass1salp}
{\mathcal{S}}_1{\mathcal{S}}_{\alpha}=\sum_{\alpha\prec_Q \beta}{\mathcal{S}}_{\beta}.
\end{align}

The following lemma is an immediate consequence of Lemma \ref{PieriruqSchurf}.
\begin{lemma}\label{lem:defprecq}
Let $\alpha$ and $\beta$ be compositions. Then $\alpha\prec_Q \beta$ if and only if $\alpha=\text{rem}_s(\beta)$ where $s$ is a part of $\beta$.
\end{lemma}

The partial orders $\leq_{C}$, $\leq_M$, $\leq_F$ and $\leq_Q$ satisfy the following relations.
\begin{lemma}\label{lem:partialorderforpo}
The binary relations $\leq_{C}$, $\leq_M$, $\leq_F$ and $\leq_Q$  satisfy $\leq_{C}\subseteq\leq_M$ and $\leq_Q\subseteq\leq_M\subseteq\leq_F$.
\end{lemma}
\begin{proof}
By definition, $\prec_C\subseteq\prec_M\subseteq\prec_F$  clearly holds, so we only need to show that $\prec_Q\subseteq\prec_M$. Let $\alpha=(\alpha_1,\alpha_2,\cdots,\alpha_k),\beta=(\beta_1,\beta_2,\cdots,\beta_t)$.
If $\alpha\prec_Q\beta$, then, by Lemma \ref{lem:defprecq}, $\alpha=\text{rem}_{\beta_j}(\beta)$ for some $j$ with $1\leq j\leq t$.
Without loss of generality, we may assume that $\beta_j$ is the rightmost part of size $\beta_j$ in $\beta$.
Thus,
\begin{align*}
\alpha=\begin{cases}
(\beta_1,\cdots,\beta_j-1,\cdots,\beta_t)& if\ \beta_j\geq2,\\
(\beta_1,\cdots,\beta_{j-1},\beta_{j+1},\cdots,\beta_t)& if\ \beta_j=1,
\end{cases}
\end{align*}
or equivalently  either $\beta=(\alpha_1,\cdots, \alpha_j+1,\cdots,\alpha_k)$ or
$\beta=(\alpha_1,\cdots, \alpha_{j-1},1,\alpha_j,\cdots,\alpha_k)$, whence $\alpha\prec_M\beta$
and hence $\prec_Q\subseteq\prec_M$.
\end{proof}

It is clear that $\mathcal{L}_{C}$, $\mathcal{Q}_{C}$, $\Gamma$ and $\mathcal{F}$ are all graded by $rank(\alpha)=|\alpha|$ and have a unique minimal element $\emptyset$.
Figures $1-4$ show four levels, i.e., ranks $1,2,3,4$, of $(\mathcal{L}_{C},\leq_{C})$, $(\mathcal{Q}_{C},\leq_Q)$,
$(\Gamma,\leq_M)$ and $(\mathcal{F},\leq_F)$, respectively.

\begin{center}
\setlength{\unitlength}{2mm}
\qquad\qquad\qquad \begin{picture}(12,12)
\linethickness{0.5pt}
\put(-9,-4){\scriptsize{Figure $1$: The poset $(\mathcal{L}_{C},\leq_{C})$}}
\put(0,0){\circle*{0.5}}\put(-0.3,-2){\scriptsize{1}}
\put(0,0){\line(3,1){6}}
\put(0,0){\line(-3,1){6}}
\put(-6,2){\circle*{0.5}}\put(-9.5,1.5){\scriptsize{11}}
\put(-6,2){\line(-1,1){3}}
\put(-6,2){\line(3,1){9}}
\put(6,2){\circle*{0.5}}\put(8,1.5){\scriptsize{2}}
\put(6,2){\line(-3,1){9}}
\put(6,2){\line(1,1){3}}
\put(-9,5){\circle*{0.5}}\put(-12.5,4.5){\scriptsize{111}}
\put(-9,5){\line(-1,2){2}}
\put(-9,5){\line(2,1){8}}
\put(-3,5){\circle*{0.5}}\put(-6,4.5){\scriptsize{12}}
\put(-3,5){\line(-5,4){5}}
\put(-3,5){\line(1,1){4}}
\put(-3,5){\line(2,1){8}}
\put(3,5){\circle*{0.5}}\put(5,4.5){\scriptsize{21}}
\put(3,5){\line(-2,1){8}}
\put(3,5){\line(1,2){2}}
\put(3,5){\line(5,4){5}}
\put(9,5){\circle*{0.5}}\put(10.5,4.5){\scriptsize{3}}
\put(9,5){\line(1,2){2}}
\put(9,5){\line(-2,1){8}}
\put(-11,9){\circle*{0.5}}\put(-14,9.5){\scriptsize{1111}}
\put(-8,9){\circle*{0.5}}\put(-9.5,9.5){\scriptsize{112}}
\put(-5,9){\circle*{0.5}}\put(-6.2,9.5){\scriptsize{121}}
\put(-1,9){\circle*{0.5}}\put(-3,9.5){\scriptsize{211}}
\put(1,9){\circle*{0.5}}\put(1,9.5){\scriptsize{13}}
\put(5,9){\circle*{0.5}}\put(4,9.5){\scriptsize{22}}
\put(8,9){\circle*{0.5}}\put(7.5,9.5){\scriptsize{31}}
\put(11,9){\circle*{0.5}}\put(11.5,9.5){\scriptsize{4}}
\end{picture} \qquad\qquad\qquad\qquad\qquad\qquad
\begin{picture}(12,12)
\linethickness{0.5pt}
\put(-9,-4){\scriptsize{Figure $2$: The poset $(\mathcal{Q}_{C},\leq_Q)$}}
\put(0,0){\circle*{0.5}}\put(-0.3,-2){\scriptsize{1}}
\put(0,0){\line(3,1){6}}
\put(0,0){\line(-3,1){6}}
\put(-6,2){\circle*{0.5}}\put(-9.5,1.5){\scriptsize{11}}
\put(-6,2){\line(-1,1){3}}
\put(-6,2){\line(1,1){3}}
\put(-6,2){\line(3,1){9}}
\put(6,2){\circle*{0.5}}\put(8,1.5){\scriptsize{2}}
\put(6,2){\line(-3,1){9}}
\put(6,2){\line(-1,1){3}}
\put(6,2){\line(1,1){3}}
\put(-9,5){\circle*{0.5}}\put(-12.5,4.5){\scriptsize{111}}
\put(-9,5){\line(-1,2){2}}
\put(-9,5){\line(1,4){1}}
\put(-9,5){\line(1,1){4}}
\put(-9,5){\line(2,1){8}}
\put(-3,5){\circle*{0.5}}\put(-6,4.5){\scriptsize{12}}
\put(-3,5){\line(-5,4){5}}
\put(-3,5){\line(-1,2){2}}
\put(-3,5){\line(1,1){4}}
\put(3,5){\circle*{0.5}}\put(5,4.5){\scriptsize{21}}
\put(3,5){\line(-1,1){4}}
\put(3,5){\line(1,2){2}}
\put(3,5){\line(5,4){5}}
\put(9,5){\circle*{0.5}}\put(10.5,4.5){\scriptsize{3}}
\put(9,5){\line(-1,4){1}}
\put(9,5){\line(1,2){2}}
\put(9,5){\line(-2,1){8}}
\put(-11,9){\circle*{0.5}}\put(-14,9.5){\scriptsize{1111}}
\put(-8,9){\circle*{0.5}}\put(-9.5,9.5){\scriptsize{112}}
\put(-5,9){\circle*{0.5}}\put(-6.2,9.5){\scriptsize{121}}
\put(-1,9){\circle*{0.5}}\put(-3,9.5){\scriptsize{211}}
\put(1,9){\circle*{0.5}}\put(1,9.5){\scriptsize{13}}
\put(5,9){\circle*{0.5}}\put(4,9.5){\scriptsize{22}}
\put(8,9){\circle*{0.5}}\put(7.5,9.5){\scriptsize{31}}
\put(11,9){\circle*{0.5}}\put(11.5,9.5){\scriptsize{4}}
\end{picture}
\end{center}
\vspace{8mm}
\begin{center}
\setlength{\unitlength}{2mm}
\qquad\qquad\qquad \begin{picture}(12,12)
\linethickness{0.5pt}
\put(-9,-4){\scriptsize{Figure $3$: The poset $(\Gamma,\leq_M)$}}
\put(0,0){\circle*{0.5}}\put(-0.3,-2){\scriptsize{1}}
\put(0,0){\line(3,1){6}}
\put(0,0){\line(-3,1){6}}
\put(-6,2){\circle*{0.5}}\put(-9.5,1.5){\scriptsize{11}}
\put(-6,2){\line(-1,1){3}}
\put(-6,2){\line(1,1){3}}
\put(-6,2){\line(3,1){9}}
\put(6,2){\circle*{0.5}}\put(8,1.5){\scriptsize{2}}
\put(6,2){\line(-3,1){9}}
\put(6,2){\line(-1,1){3}}
\put(6,2){\line(1,1){3}}
\put(-9,5){\circle*{0.5}}\put(-12.5,4.5){\scriptsize{111}}
\put(-9,5){\line(-1,2){2}}
\put(-9,5){\line(1,4){1}}
\put(-9,5){\line(1,1){4}}
\put(-9,5){\line(2,1){8}}
\put(-3,5){\circle*{0.5}}\put(-6,4.5){\scriptsize{12}}
\put(-3,5){\line(-5,4){5}}
\put(-3,5){\line(-1,2){2}}
\put(-3,5){\line(1,1){4}}
\put(-3,5){\line(2,1){8}}
\put(3,5){\circle*{0.5}}\put(5,4.5){\scriptsize{21}}
\put(3,5){\line(-2,1){8}}
\put(3,5){\line(-1,1){4}}
\put(3,5){\line(1,2){2}}
\put(3,5){\line(5,4){5}}
\put(9,5){\circle*{0.5}}\put(10.5,4.5){\scriptsize{3}}
\put(9,5){\line(-1,4){1}}
\put(9,5){\line(1,2){2}}
\put(9,5){\line(-2,1){8}}
\put(-11,9){\circle*{0.5}}\put(-14,9.5){\scriptsize{1111}}
\put(-8,9){\circle*{0.5}}\put(-9.5,9.5){\scriptsize{112}}
\put(-5,9){\circle*{0.5}}\put(-6.2,9.5){\scriptsize{121}}
\put(-1,9){\circle*{0.5}}\put(-3,9.5){\scriptsize{211}}
\put(1,9){\circle*{0.5}}\put(1,9.5){\scriptsize{13}}
\put(5,9){\circle*{0.5}}\put(4,9.5){\scriptsize{22}}
\put(8,9){\circle*{0.5}}\put(7.5,9.5){\scriptsize{31}}
\put(11,9){\circle*{0.5}}\put(11.5,9.5){\scriptsize{4}}
\end{picture}
\qquad\qquad\qquad\qquad\qquad\qquad
\begin{picture}(12,12)
\linethickness{0.5pt}
\put(-9,-4){\scriptsize{Figure $4$: The poset $(\mathcal{F},\leq_F)$}}
\put(0,0){\circle*{0.5}}\put(-0.3,-2){\scriptsize{1}}
\put(0,0){\line(3,1){6}}
\put(0,0){\line(-3,1){6}}
\put(-6,2){\circle*{0.5}}\put(-9.5,1.5){\scriptsize{11}}
\put(-6,2){\line(-1,1){3}}
\put(-6,2){\line(1,1){3}}
\put(-6,2){\line(3,1){9}}
\put(6,2){\circle*{0.5}}\put(8,1.5){\scriptsize{2}}
\put(6,2){\line(-3,1){9}}
\put(6,2){\line(-1,1){3}}
\put(6,2){\line(1,1){3}}
\put(-9,5){\circle*{0.5}}\put(-12.5,4.5){\scriptsize{111}}
\put(-9,5){\line(-1,2){2}}
\put(-9,5){\line(1,4){1}}
\put(-9,5){\line(1,1){4}}
\put(-9,5){\line(2,1){8}}
\put(-3,5){\circle*{0.5}}\put(-6,4.5){\scriptsize{12}}
\put(-3,5){\line(-5,4){5}}
\put(-3,5){\line(-1,2){2}}
\put(-3,5){\line(1,1){4}}
\put(-3,5){\line(2,1){8}}
\put(3,5){\circle*{0.5}}\put(5,4.5){\scriptsize{21}}
\put(3,5){\line(-2,1){8}}
\put(3,5){\line(-1,1){4}}
\put(3,5){\line(1,2){2}}
\put(3,5){\line(5,4){5}}
\put(9,5){\circle*{0.5}}\put(10.5,4.5){\scriptsize{3}}
\put(9,5){\line(-1,4){1}}
\put(9,5){\line(-1,1){4}}
\put(9,5){\line(1,2){2}}
\put(9,5){\line(-2,1){8}}
\put(-11,9){\circle*{0.5}}\put(-14,9.5){\scriptsize{1111}}
\put(-8,9){\circle*{0.5}}\put(-9.5,9.5){\scriptsize{112}}
\put(-5,9){\circle*{0.5}}\put(-6.2,9.5){\scriptsize{121}}
\put(-1,9){\circle*{0.5}}\put(-3,9.5){\scriptsize{211}}
\put(1,9){\circle*{0.5}}\put(1,9.5){\scriptsize{13}}
\put(5,9){\circle*{0.5}}\put(4,9.5){\scriptsize{22}}
\put(8,9){\circle*{0.5}}\put(7.5,9.5){\scriptsize{31}}
\put(11,9){\circle*{0.5}}\put(11.5,9.5){\scriptsize{4}}
\end{picture}
\end{center}

\vspace{10mm}

Let $\leq$ be a partial order on the set $\mathcal{C}$, and let $\beta\in \mathcal{C}$ be a composition. We denote $D_{\leq}(\beta)$
the set of compositions covered by $\beta$. In symbols,
$$
D_{\leq}(\beta)=\{\alpha\in \mathcal{C}|\alpha\prec \beta\}.
$$

\begin{lemma}\label{lem:weightandlengthofalbe}
Let $\leq$ be one of the partial orders $\leq_Q,\leq_M,\leq_F$, and let $\alpha,\beta$ be compositions with $|\alpha|\geq3$.
If $D_{\leq}(\alpha)=D_{\leq}(\beta)$, then $|\alpha|=|\beta|$ and $\ell(\alpha)=\ell(\beta)$.
\end{lemma}
\begin{proof}
We can see that $|\alpha|=|\beta|$ since $(\mathcal{C},\leq)$ is a graded poset.
If one of $\alpha,\beta$  is $(1^n)$  where $n\geq3$, without loss of generality, suppose that $\alpha=(1^n)$, then $D_{\leq}(\alpha)=\{(1^{n-1})\}$.
Thus, we must have $\beta=(1^n)$ so that $\alpha=\beta$ and hence $\ell(\alpha)=\ell(\beta)$.

In what follows, assume that neither $\alpha$ nor  $\beta$ is equal to $(1^n)$.
Let $\alpha=(\alpha_1,\alpha_2,\cdots,\alpha_k)$.
Then there exists a part of $\alpha$ larger than $1$. Let
$\alpha_j$ be the rightmost part of $\alpha$  such that $\alpha_j\geq2$ and denote by $\delta=(\alpha_1,\cdots,\alpha_j-1,\cdots,\alpha_k)$. Then $\ell(\delta)=\ell(\alpha)$.
Note that $\delta=\text{rem}_{\alpha_j}(\alpha)\prec_Q\alpha$, so that $\delta\prec_M\alpha$ and $\delta\prec_F\alpha$ by Lemma \ref{lem:partialorderforpo},
whence $\delta$ is in the set $D_{\leq}(\alpha)=D_{\leq}(\beta)$ for any partial order $\leq$ in $\{\leq_Q,\leq_M,\leq_F\}$.
By definition, we have either
$\ell(\delta)=\ell(\beta)$ or $\ell(\delta)=\ell(\beta)-1$, and hence $\ell(\alpha)\leq \ell(\beta)$.
By symmetry, we also have $\ell(\beta)\leq \ell(\alpha)$, proving that $\ell(\alpha)= \ell(\beta)$.
\end{proof}

\begin{lemma}\label{lem:D(al)coversamelength}
Let $\leq$ be one of the partial orders $\leq_Q,\leq_M,\leq_F$, and let $\alpha$, $\beta$ be compositions with the same length,
where $\alpha=(\alpha_1,\alpha_2,\cdots,\alpha_k)$.
\begin{enumerate}
\item\label{lem:itemalneqbe+-1} If $\alpha\neq \beta$ and $D_{\leq}(\alpha)\cap D_{\leq}(\beta)$ contains an element of length $\ell(\alpha)$,
then there exist $i,j$ with $1\leq i<j\leq k$ such that either $\beta=(\alpha_1,\cdots, \alpha_i+1,\cdots,\alpha_j-1,\cdots,\alpha_k)$ or $\beta=(\alpha_1,\cdots, \alpha_i-1,\cdots,\alpha_j+1,\cdots,\alpha_k)$;
\item\label{lem:itemdal=dbeimpal=be} If $D_{\leq}(\alpha)\cap D_{\leq}(\beta)$ contains two distinct elements of length $\ell(\alpha)$, then $\alpha=\beta$.
\end{enumerate}
\end{lemma}
\begin{proof}
\eqref{lem:itemalneqbe+-1} Since $\alpha, \beta$ have the same length, it follows from $\alpha\neq\beta$ that $k\geq 2$.
Let $\sigma\in D_{\leq}(\alpha)\cap D_{\leq}(\beta)$
be an element of length $\ell(\alpha)$. By definition, $\sigma$ must be obtained from $\alpha$ and $\beta$ by subtracting $1$
from some part in $\alpha$ and $\beta$, respectively. Thus, there exist distinct $i,j$ with $1\leq i,j\leq k$ such that
\begin{align*}
    \sigma=(\alpha_1,\cdots,\alpha_i-1,\cdots,\alpha_k)=(\beta_1,\cdots,\beta_j-1,\cdots,\beta_k).
\end{align*}
Hence part \eqref{lem:itemalneqbe+-1} follows.

\eqref{lem:itemdal=dbeimpal=be} Let $\sigma,\delta\in D_{\leq}(\alpha)\cap D_{\leq}(\beta)$
be two distinct elements of length $\ell(\alpha)$.
Analogous to the proof of part \eqref{lem:itemalneqbe+-1}, we see that
there exist $i,j$, $s,t$ such that
\begin{align*}
    \sigma=(\alpha_1,\cdots,\alpha_i-1,\cdots,\alpha_k)=(\beta_1,\cdots,\beta_j-1,\cdots,\beta_k)
\end{align*}
and
\begin{align*}
    \delta=(\alpha_1,\cdots,\alpha_s-1,\cdots,\alpha_k)=(\beta_1,\cdots,\beta_t-1,\cdots,\beta_k).
\end{align*}
If $\alpha\neq\beta$, then $i\neq j$ and $s\neq t$. Suppose without loss of generality that $i<j$.  If $s<t$, then
\begin{align*}
    \beta=(\alpha_1,\cdots,\alpha_i-1,\cdots,\alpha_j+1,\cdots,\alpha_k)=(\alpha_1,\cdots,\alpha_s-1,\cdots,\alpha_t+1,\cdots,\alpha_k),
\end{align*}
so that $i=s$, $j=t$ and hence $\sigma=\delta$, a contradiction. If $s>t$, then we can get a contradiction in a similar way. This completes the proof.
\end{proof}

\begin{proposition}\label{lem:al=beccoverset=}
Let $\leq$ be one of the partial orders $\leq_M$ and $\leq_F$, and let $\alpha,\beta$ be compositions with $|\alpha|\geq4$. If $D_{\leq}(\alpha)=D_{\leq}(\beta)$ then $\alpha=\beta$.
\end{proposition}
\begin{proof}
It follows from Lemma \ref{lem:weightandlengthofalbe} and $D_{\leq}(\alpha)=D_{\leq}(\beta)$  that $|\alpha|=|\beta|$ and $\ell(\alpha)=\ell(\beta)$.
It is trivial to see that $\alpha=\beta$
if one of $\alpha$ and $\beta$  is $(1^n)$ or $(n)$, where $n=|\alpha|\geq4$.
If there exist two parts in $\alpha$ or $\beta$ are larger than $1$, then $D_{\leq}(\alpha)=D_{\leq}(\beta)$ contains at least two distinct
elements of length $\ell(\alpha)$.
By Lemma \ref{lem:D(al)coversamelength}\eqref{lem:itemdal=dbeimpal=be}, we have $\alpha=\beta$.
Thus, we next
assume that $\alpha=(1^{i},a,1^{n-i-a})$ and $\beta=(1^{j},b,1^{n-j-b})$, where $0\leq i\leq n-a$, $0\leq j\leq n-b$, $2\leq a\leq n-1$, $2\leq b\leq n-1$.

Suppose to the contrary that $\alpha\neq\beta$. Note that $(1^{i},a-1,1^{n-i-a})$ is an element of length $\ell(\alpha)$ contained in $D_{\leq}(\alpha)=D_{\leq}(\beta)$.
By Lemma \ref{lem:D(al)coversamelength}\eqref{lem:itemalneqbe+-1} we have $a=b=2$, whence $i\neq j$.
 Without loss of generality we may assume that $i<j$.
Then $(1^{j-1},2,1^{n-j-2})$ is covered by $\beta$ so that it belongs to $D_{\leq}(\alpha)$, which yields $j=i+1$.
Hence $\alpha=(1^{i},2,1^{n-i-2})$ and $\beta=(1^{i+1},2,1^{n-i-3})$.
Now if $i\geq1$, then $(1^{i-1},2,1^{n-i-2})\in  D_{\leq}(\alpha)\backslash D_{\leq}(\beta)$, a contradiction; if $i=0$, then
$\alpha=(2,1^{n-2})$ and $\beta=(1,2,1^{n-3})$, which together with $n\geq4$ implies that
$(1,2,1^{n-4})\in  D_{\leq}(\beta)\backslash D_{\leq}(\alpha)$, another contradiction, completing the proof.
\end{proof}

The algebraic meaning of Proposition \ref{lem:al=beccoverset=} can be stated as follows.

\begin{corollary}\label{lemma:geq4albefundeq}
Let $\alpha$ and $\beta$ be two compositions with weight $n\geq4$.
\begin{enumerate}
\item\label{lemitem:ammmal=bemon} If
$
\langle M_\alpha,M_1 M_\delta\rangle_M=\langle M_\beta,M_1M_\delta\rangle_M
$
for all compositions $\delta$ with weight $n-1$, then $\alpha=\beta$;
\item\label{lemitem:ammmal=befunb}  If
$
\langle F_\alpha,F_1 F_\delta\rangle_F=\langle F_\beta,F_1 F_\delta\rangle_F
$
for all compositions $\delta$ with weight $n-1$, then $\alpha=\beta$.
\end{enumerate}
\end{corollary}
\begin{proof}
It follows directly from $\langle M_\alpha,M_1 M_\delta\rangle_M=\langle M_\beta,M_1M_\delta\rangle_M$ that
$\langle M_\alpha,M_1 M_\delta\rangle_M> 0$ is equivalent to $\langle M_\beta,M_1M_\delta\rangle_M>0$, which together with
Lemma \ref{lem:m1dalmonomialb} yields that $D_{\leq_M}(\alpha)=D_{\leq_M}(\beta)$. Thus, by Proposition \ref{lem:al=beccoverset=}, we obtain that $\alpha=\beta$,
completing the proof of part \eqref{lemitem:ammmal=bemon}.
By Eq.\eqref{fundbasf1falp}, an analogous argument works for part \eqref{lemitem:ammmal=befunb}.
\end{proof}

\begin{proposition}\label{lem:aln=beccoversetleqq}
Let $\alpha$ and  $\beta$ be distinct compositions with weights larger than $2$.
Then  $D_{\leq_Q}(\alpha)=D_{\leq_Q}(\beta)$ if and only if there exist nonnegative integers $i_1,\cdots,i_k,j_1,\cdots,j_k$ with $k\geq0$ such that
\begin{align}\label{eq:albepropQord}
\{\alpha,\beta\}=\{(1^{i_1},2^{j_1},\cdots,1^{i_k},2^{j_k},1,2),(1^{i_1},2^{j_1},\cdots,1^{i_k},2^{j_k+1},1)\}.
\end{align}
\end{proposition}
\begin{proof}
It follows from Eq.\eqref{eq:albepropQord} that $\text{rem}_{1}(\alpha)=\text{rem}_{1}(\beta)=(1^{i_1},2^{j_1},\cdots,1^{i_k},2^{j_k+1})$
and $\text{rem}_{2}(\alpha)=\text{rem}_{2}(\beta)=(1^{i_1},2^{j_1},\cdots,1^{i_k},2^{j_k},1^2)$. Hence $D_{\leq_Q}(\alpha)=D_{\leq_Q}(\beta)$.

Conversely, assume that $\alpha$ and  $\beta$ are distinct compositions with weights larger than $2$ and $D_{\leq_Q}(\alpha)=D_{\leq_Q}(\beta)$.
Then, by Lemma \ref{lem:weightandlengthofalbe}, $|\alpha|=|\beta|$ and $\ell(\alpha)=\ell(\beta)$.
It is easy to see that both $\alpha$ and $\beta$ have parts greater than $1$. So $D_{\leq_Q}(\alpha)\cap D_{\leq_Q}(\beta)$ contains an element of length $\ell(\alpha)$.
Let  $\alpha=(\alpha_1,\alpha_2,\cdots,\alpha_{\ell})$.
Applying Lemma \ref{lem:D(al)coversamelength}\eqref{lem:itemalneqbe+-1} we may suppose without loss of generality that   $\beta=(\alpha_1,\cdots,\alpha_i+1,\cdots,\alpha_j-1,\cdots,\alpha_{\ell})$ for some $i,j$ with $1\leq i<j\leq\ell$.

We claim that $\beta_{i+1}=\cdots=\beta_{\ell}=1$. Otherwise, we may assume that $\beta_t$ is the rightmost part of size grater than $1$ in $\beta$.
Then $t> i$, and hence
$\text{rem}_{\beta_t}(\beta)=(\beta_1,\cdots,\beta_i,\cdots,\beta_t-1,\cdots,\beta_{\ell})$ belongs to $D_{\leq_Q}(\beta)=D_{\leq_Q}(\alpha)$.
From $\beta_t>1$ we conclude that
$\ell(\text{rem}_{\beta_t}(\beta))=\ell(\alpha)$. Thus, $\text{rem}_{\beta_t}(\beta)$ is  obtained from $\alpha$ by subtracting $1$ from some part in $\alpha$, contradicting the fact that
$\beta_i=\alpha_i+1$. So we have $\beta_{i+1}=\cdots=\beta_{\ell}=1$, whence $\beta=(\alpha_1,\cdots,\alpha_i+1,1^{\ell-i})$.

From $\ell>i$ we see that $\text{rem}_{1}(\beta)=(\alpha_1,\cdots,\alpha_i+1,1^{\ell-i-1})$ belongs to $D_{\leq_Q}(\beta)=D_{\leq_Q}(\alpha)$,
so that $\ell(\text{rem}_{1}(\beta))=\ell(\alpha)-1$, and hence $\text{rem}_{1}(\beta)=\text{rem}_{1}(\alpha)$.
Since the $i$-th part of $\text{rem}_{1}(\beta)$ is $\alpha_i+1$, we have $\alpha_i=1$ and $\text{rem}_{1}(\alpha)=(\alpha_1,\cdots,\alpha_{i-1},\alpha_{i+1},\cdots,\alpha_{\ell})$.
Thus, $\ell=i+1$, whence
$\alpha=(\alpha_1,\cdots,\alpha_{\ell-2},1,2)$ and $\beta=(\alpha_1,\cdots,\alpha_{\ell-2},2,1)$.

If there exists a part of $\alpha$ larger than $2$, then we may choose the rightmost part of that size, say $\alpha_p$, where $1\leq p\leq \ell-2$, so
$\text{rem}_{\alpha_p}(\alpha)=(\alpha_1,\cdots,\alpha_p-1,\cdots,\alpha_{\ell-2},1,2)\in D_{\leq_Q}(\alpha)\backslash D_{\leq_Q}(\beta)$,
contradicting  $D_{\leq_Q}(\alpha)=D_{\leq_Q}(\beta)$. Therefore, we have $1\leq \alpha_s=\beta_s\leq2$ for all $s$ with $1\leq s\leq \ell-2$, and the result follows.
\end{proof}

\begin{proposition}\label{lembgdcovers}
Let $\alpha,\beta$ be two distinct compositions of $n$ with $\ell(\alpha)\leq \ell(\beta)$. Then
$D_{\leq_{C}}(\alpha)=D_{\leq_{C}}(\beta)$ if and only if one of the following conditions holds
\begin{enumerate}
\item\label{lem:itembetgamm211} $\alpha=(2)$, $\beta=(1,1)$;
\item\label{lem:itembetgamm312} $\alpha=(3)$, $\beta=(1,2)$;
\item\label{lem:itembetgamacoma+11a1kbleq}
 $\alpha=(a +1,1,b_1,\cdots,b_k)$, $\beta=(1,a,1,b_1,\cdots,b_k)$ where $a\in\{1,2\}$, $k\geq0$ and  $b_1,b_2,\cdots,b_k$ are positive integers satisfying
\begin{align}\label{eq:creleqbjbega}
b_{j}\leq \max\{b_{i}+1|i=0,1,\cdots,j-1\}
\end{align}
for all $1\leq j\leq k$,
with the notation $b_0=1$.
\end{enumerate}
\end{proposition}
\begin{proof}
Let $\alpha$, $\beta$ be two distinct compositions of $n$ such that
$D_{\leq_{C}}(\alpha)=D_{\leq_{C}}(\beta)$.
Since there is only one composition of $1$, we have $n\geq2$. It is routine to see that either  $\alpha=(2)$, $\beta=(1,1)$ or $\alpha=(3)$, $\beta=(1,2)$ if $\ell(\alpha)=1$.
It is clear that in these cases the sets of  covered compositions are the same. Thus we are in case \eqref{lem:itembetgamm211} or \eqref{lem:itembetgamm312}.
We next assume that $\alpha=(\alpha_1,\alpha_2,\cdots,\alpha_{\ell(\alpha)})$ with $\ell(\alpha)\geq2$.

If $\alpha_1=1$, then the composition $\mu=(\alpha_2,\cdots,\alpha_{\ell(\alpha)})$ is covered by $\alpha$
so that $\mu\prec_C\beta$.
But now it follows from $\alpha\neq\beta$ that $\beta=(\alpha_2,\cdots,\alpha_s+1,\cdots,\alpha_{\ell(\alpha)})$
for some $s$ with $2\leq s\leq \ell(\alpha)$, so that $\ell(\beta)=\ell(\alpha)-1$, contradicting
$\ell(\alpha)\leq\ell(\beta)$. Hence $\alpha_1\geq2$.

Let $\nu=(\alpha_1-1,\alpha_2,\cdots,\alpha_{\ell(\alpha)})$. Then $\nu\prec_C\alpha$, whence $\nu\prec_C\beta$.
So $\beta$ is either $(1,\alpha_1-1,\alpha_2,\cdots,\alpha_{\ell(\alpha)})$ or $(\alpha_1-1,\alpha_2,\cdots,\alpha_t+1,\cdots,\alpha_{\ell(\alpha)})$ for some
$t\geq2$ and
$\alpha_t\not\in\{\alpha_1-1,\alpha_2,\cdots,\alpha_{t-1}\}$.
We next show that the second case will never happen. Otherwise, suppose that $\beta=(\alpha_1-1,\alpha_2,\cdots,\alpha_t+1,\cdots,\alpha_{\ell(\alpha)})$.
If $\alpha_1=2$, then the first part of $\beta$ is $1$, so we have
$$
(\alpha_2,\cdots,\alpha_t+1,\cdots,\alpha_{\ell(\alpha)})\in D_{\leq_{C}}(\beta) \backslash D_{\leq_{C}}(\alpha),
$$
a contradiction. If $\alpha_1\geq3$, then we have
$$
(\alpha_1-2,\alpha_2,\cdots,\alpha_t+1,\cdots,\alpha_{\ell(\alpha)}) \in D_{\leq_{C}}(\beta) \backslash D_{\leq_{C}}(\alpha),
$$
another contradiction.
Thus, we must have $\beta=(1,\alpha_1-1,\alpha_2,\cdots,\alpha_{\ell(\alpha)})$.

We now show that $\alpha_2=1$. Otherwise, $\alpha_2\geq2$.
Now if $\alpha_2-1=\alpha_1$, then
$$
(1,\alpha_1-1,\alpha_2-1,\alpha_3,\cdots,\alpha_{\ell(\alpha)})\in D_{\leq_{C}}(\beta) \backslash D_{\leq_{C}}(\alpha);
$$
if $\alpha_2-1\neq\alpha_1$, then
$$(\alpha_1,\alpha_2-1,\alpha_3,\cdots,\alpha_{\ell(\alpha)})\in D_{\leq_{C}}(\alpha) \backslash D_{\leq_{C}}(\beta),
$$
contradicting $ D_{\leq_{C}}(\alpha)=D_{\leq_{C}}(\beta)$.
 Hence $\alpha_2=1$.

Take $k=\ell(\alpha)-2$ and let $a=\alpha_1-1$, $b_i=\alpha_{i+2}$ for $1\leq i\leq k$.
Then  $k\geq0$ and
$$\alpha=(a +1,1,b_1,\cdots,b_{k}),\quad \beta=(1,a,1,b_1,\cdots,b_{k}).$$
 If $a\geq3$, then $a-1\neq1$ and hence
 $$
(1,a-1,1,b_1,\cdots,b_{k})\in D_{\leq_{C}}(\beta) \backslash D_{\leq_{C}}(\alpha),
$$
contradicting $D_{\leq_{C}}(\alpha) =D_{\leq_{C}}(\beta)$.
We thus obtain that  $a\in \{1,2\}$.

Assume that there exists $j$ with $1\leq j\leq k$ such that $b_{j}> \max\{b_{i}+1|i=0,1,\cdots,j-1\}$.
Then $b_{j}-1> \max\{b_{i}|i=0,1,\cdots,j-1\}$.  If $b_{j}-1\neq a+1$, then we have
$$
(a+1,1,b_1,\cdots,b_{j}-1,\cdots,b_k)\in D_{\leq_{C}}(\alpha) \backslash D_{\leq_{C}}(\beta),
$$
a contradiction.
If $b_{j}-1= a+1$, then $b_{j}-1\neq a$, so
$$
(1,a,1,b_1,\cdots,b_{j}-1,\cdots,b_k)\in D_{\leq_{C}}(\beta) \backslash D_{\leq_{C}}(\alpha),
$$
another contradiction,
proving Eq.\eqref{eq:creleqbjbega}, and the proof of part \eqref{lem:itembetgamacoma+11a1kbleq} follows.

Conversely, let $\alpha,\beta$ be the compositions given by part
\eqref{lem:itembetgamacoma+11a1kbleq}, and let $\gamma=(a,1,b_1,\cdots,b_k)$.
We shall show that $D_{\leq_{C}}(\alpha)=D_{\leq_{C}}(\beta) =\{\gamma\}$.
Obviously, $\gamma$ belongs to $D_{\leq_{C}}(\alpha)$ and $D_{\leq_{C}}(\beta)$.

Take an element $\sigma\in D_{\leq_{C}}(\alpha)$. Since $a\in\{1,2\}$, the first part of $\alpha$, i.e.,  $a+1$, is larger than $1$. Then, by Definition \ref{reversecompposet}, $\alpha$ can be only obtained from $\sigma$  by adding $1$ to the first part of $\sigma$ of a given size.
If $\sigma\neq\gamma$, then there exists $j$ with $1\leq j\leq k$ such that
$\sigma=(a+1,1,b_1,\cdots,b_j-1,\cdots,b_k)$. It follows from Eq.\eqref{eq:creleqbjbega} that $b_j-1\in\{1,b_1,\cdots,b_{j-1}\}$,
contradicting Definition \ref{reversecompposet}. Hence $\sigma=\gamma$, so that $D_{\leq_{C}}(\alpha) =\{\gamma\}$.

The proof of $D_{\leq_{C}}(\beta) =\{\gamma\}$ is analogous to that of $D_{\leq_{C}}(\alpha) =\{\gamma\}$.
Take any element $\delta\in D_{\leq_{C}}(\beta)$.
If $\delta\neq\gamma$, then $\beta$ can be obtained from $\delta$
by adding $1$ to the first part of $\delta$ of a given size.
So we deduce from $a\in\{1,2\}$ that  there exists $j$ with $1\leq j\leq k$ such that
$\delta=(1,a,1,b_1,\cdots,b_j-1,\cdots,b_k)$. However, Eq.\eqref{eq:creleqbjbega}  guarantees  that $b_j-1\in\{1,b_1,\cdots,b_{j-1}\}$, a contradiction.
Thus, $\delta=\gamma$, and hence $D_{\leq_{C}}(\beta) =\{\gamma\}$,
proving
$D_{\leq_{C}}(\alpha)=D_{\leq_{C}}(\beta)$.
\end{proof}


\section{Rigidity with respect to the monomial quasisymmetric basis}\label{sec:rwrttmb}

In this section, we show that the involutive map $\rho$ defined by Eq.\eqref{iafcom} is the unique graded algebra automorphism preserving the monomial basis.
But there is no such a coalgebra automorphism, so ${\rm {\Qsym}}$ is rigid as a Hopf algebra with respect to the monomial basis.
Before we do this, we will work towards two
lemmas.

\begin{lemma}\label{lemma:identitymapgrdaed}
Let $K\in\{M,F,\mathcal{S}\}$, and let $\varphi$  be a graded bijection of ${\rm {\Qsym}}$ that preserves the basis $\{K_{\alpha}|\alpha\in\mathcal{C}\}$.
Then $\varphi(K_{\emptyset})=K_{\emptyset}$, $\varphi(K_{1})=K_{1}$, and there are two possibilities for degree $2$, i.e.,
either $\varphi(K_{11})=K_{11}$, $\varphi(K_{2})=K_{2}$ or
 $\varphi(K_{11})=K_{2}$, $\varphi(K_{2})=K_{11}$.
\end{lemma}
\begin{proof}
This is trivial.
\end{proof}

\begin{lemma}\label{lemma:alghomalbevaralbe}
Let $K\in\{M,F\}$, and let $\varphi$ and $\phi$ be two graded algebra automorphisms of ${\rm {\Qsym}}$ that preserve the basis $\{K_{\alpha}|\alpha\in\mathcal{C}\}$ where $\phi^2=\phi$.
If $\varphi(K_{\alpha})=\phi(K_{\alpha})$  for all compositions $\alpha$ with $|\alpha|\leq3$, then $\varphi=\phi$.
\end{lemma}
\begin{proof}
Take an arbitrary composition $\alpha$. We shall show  by induction on $|\alpha|$ that $\varphi(K_{\alpha})=\phi(K_{\alpha})$.
This is obvious if $|\alpha|\leq3$. Supposing by induction
that $\varphi(K_{\sigma})=\phi(K_{\sigma})$ for all compositions $\sigma$ with $|\sigma|=|\alpha|-1$, where $|\alpha|\geq4$.
Since $\varphi$ is an algebra automorphism of ${\rm {\Qsym}}$, we have
\begin{align*}
\langle \varphi( K_{\alpha}),\phi(K_{1})\phi(K_{\delta})\rangle_K
=\langle \varphi (K_{\alpha}),\varphi (K_{1})\varphi (K_{\delta})\rangle_K
=\langle K_{\alpha},K_{1}K_{\delta}\rangle_K
\end{align*}
for all $\delta$ with $|\delta|=|\alpha|-1$.
Note that $\phi^2=\phi$ and that $\phi$ is an algebra automorphism, we have
\begin{align*}
\langle \phi\varphi( K_{\alpha}),K_{1}K_{\delta}\rangle_K
=\langle \phi (K_{\alpha}),\phi (K_{1})\phi (K_{\delta})\rangle_K
=\langle K_{\alpha},K_{1}K_{\delta}\rangle_K.
\end{align*}
Since $\varphi$ and $\phi$ preserve the basis $\{K_{\alpha}|\alpha\in\mathcal{C}\}$, $\phi\varphi( K_{\alpha})$ is in $\{K_{\alpha}|\alpha\in\mathcal{C}\}$.
By Corollary \ref{lemma:geq4albefundeq}, we have $\phi\varphi( K_{\alpha})=K_\alpha$,
and hence  $\varphi(K_{\alpha})=\phi(K_{\alpha})$.
\end{proof}

\begin{lemma}\label{lemma:111=12213}
If $\varphi$ is a graded algebra automorphism of ${\rm {\Qsym}}$ that preserves the monomial basis, then
\begin{enumerate}
\item
$\varphi(M_\alpha)=M_\alpha$ for all $\alpha\in\{(1),(1,1),(2),(1,1,1),(3)\}$;
\item either $\varphi(M_{12})=M_{12}$, $\varphi(M_{21})=M_{21}$ or $\varphi(M_{12})=M_{21}$, $\varphi(M_{21})=M_{12}$.
\end{enumerate}
\end{lemma}
\begin{proof}
By Lemma \ref{lemma:identitymapgrdaed}, we have $\varphi(M_1)=M_1$, so that $\varphi(M_1)\varphi(M_1)=M_1M_1=2M_{11}+M_2$.
Note that $\varphi$ is an algebra automorphism, we have $\langle\varphi(M_{11}),\varphi(M_1)\varphi(M_1)\rangle_M=\langle M_{11},M_1M_1\rangle_M=2$,
so that $\varphi(M_{11})=M_{11}$.
In an analogous manner we can show that $\varphi(M_2)=M_2$.

An analogous argument applying to
$$\varphi(M_{1}M_{11})=\varphi(M_1)\varphi(M_{11})=M_1M_{11}=3M_{111}+M_{12}+M_{21}$$
yields that
$\varphi(M_{111})=M_{111}$,  and $\varphi(M_{12})+\varphi(M_{21})=M_{12}+M_{21}$,
which means that either $\varphi(M_{12})=M_{12}$,
$\varphi(M_{21})=M_{21}$ or $\varphi(M_{12})=M_{21}$, $\varphi(M_{21})=M_{12}$.
The remaining basis element of degree $3$ is $M_3$, and hence $\varphi(M_3)=M_3$.
The proof follows.
\end{proof}

We now show that the map $\rho$ preserves the monomial basis.
\begin{lemma}\label{lemma:rhomonbalha1}
For any composition $\alpha$, we have $\rho(M_\alpha)=M_{\alpha^r}$.
\end{lemma}
\begin{proof}
By Eqs. \eqref{iafcom} and \eqref{eq:MFbas}, we have
\begin{align*}
    \rho(M_{\alpha})=\sum_{\beta\preceq\alpha}(-1)^{\ell(\beta)-\ell(\alpha)}F_{\beta^r}.
\end{align*}
Observe that $\ell(\alpha)=\ell(\alpha^r)$, $\ell(\beta)=\ell(\beta^r)$, and that $\beta\preceq\alpha$ is equivalent to $\beta^r\preceq\alpha^r$.
Thus it follows at once from Eq.\eqref{iafcom} that
\begin{align*}
    \rho(M_{\alpha})=\sum_{\beta^r\preceq\alpha^r}(-1)^{\ell(\beta^r)-\ell(\alpha^r)}F_{\beta^r}=M_{\alpha^r},
\end{align*}
and we are done.
\end{proof}

\begin{proposition}\label{lemma:kcalgautoast}
The map $\rho$ is
the unique nontrivial graded algebra automorphism of ${\rm {\Qsym}}$ that takes the monomial basis into itself.
\end{proposition}

\begin{proof}
By \cite[page 50]{M93} and
Lemma \ref{lemma:rhomonbalha1}, the map $\rho$  is a nontrivial graded algebra automorphism that takes the monomial basis into itself.
Conversely, we need to show that any nontrivial graded algebra automorphism $\varphi$ of ${\rm {\Qsym}}$ preserving the monomial basis must coincide with
$\rho$.

By Lemma \ref{lemma:111=12213},  there are two possibilities for the values of $\varphi$ on compositions of weight $\leq3$:
either
$\varphi(M_\alpha)=M_\alpha$ for all $\alpha$ with $|\alpha|\leq 3$ or
$\varphi(M_\alpha)=M_{\alpha^r}$ for all $\alpha$ with $|\alpha|\leq3$. By Lemma \ref{lemma:alghomalbevaralbe},
$\varphi$ is the identity map in the former case,
and $\varphi=\rho$ in the second case.
Therefore, $\varphi=\rho$ if $\varphi$ is nontrivial,
completing the proof.
\end{proof}

\begin{proposition}\label{lemma:kcoalgautoast}
There are no nontrivial graded coalgebra automorphisms of ${\rm {\Qsym}}$ that take the monomial basis into itself.
\end{proposition}

\begin{proof}
Let $\varphi$ be a graded coalgebra automorphism of ${\rm {\Qsym}}$ that takes the monomial basis into itself. Then for any composition $\alpha$
we have $(\varphi\otimes \varphi)(\Delta(M_{\alpha}))=\Delta(\varphi(M_\alpha))$.

We prove by induction on $|\alpha|$, the weight of $\alpha$, that $\varphi(M_\alpha)=M_\alpha$. The case  $|\alpha|=1$ is straightforward and hence  omitted.
Now suppose that  $\varphi(M_\alpha)=M_\alpha$ for all compositions $\alpha$ with $|\alpha|\leq n-1$.

Let $\alpha=(\alpha_1,\alpha_2,\cdots,\alpha_k)$ be a composition of $n$.
If $k=1$, then $\alpha=(n)$, and  $\Delta(\varphi(M_{\alpha}))=(\varphi\otimes \varphi)(\Delta(M_{\alpha}))=1\otimes \varphi(M_{\alpha})+\varphi(M_{\alpha})\otimes 1$, whence $\varphi(M_\alpha)=M_\alpha$.
If $k\geq2$, then, by hypothesis,   $M_{(\alpha_1)}\otimes M_{(\alpha_2,\cdots,\alpha_k)}=\varphi(M_{(\alpha_1)})\otimes \varphi(M_{(\alpha_2,\cdots,\alpha_k)})$,
 which is a term of $(\varphi\otimes \varphi)(\Delta(M_{\alpha}))$. So $M_{(\alpha_1)}\otimes M_{(\alpha_2,\cdots,\alpha_k)}$
 is a term of $\Delta(\varphi(M_\alpha))$ and hence $\varphi(M_\alpha)=M_\alpha$, as required.
\end{proof}

With these results at hand, we can now embark on the main theorem.
\begin{theorem}\label{monomicomalone}
${\rm {\Qsym}}$ is rigid as a Hopf algebra with respect to the monomial basis, i.e., there are no nontrivial graded Hopf algebra automorphisms preserving  the set of monomial quasisymmetric functions.
\end{theorem}
\begin{proof}
By Proposition \ref{lemma:kcalgautoast}, $\rho$ is the only possible nontrivial graded algebra automorphism of  ${\rm {\Qsym}}$ that preserves the set of monomial quasisymmetric functions.
However, by Proposition \ref{lemma:kcoalgautoast}, $\rho$ is not a coalgebra automorphism. Thus, there are no nontrivial desired graded Hopf algebra automorphisms.
\end{proof}


\section{Rigidity with respect to the fundamental quasisymmetric basis}\label{sec:rwrttfb}

In this section, we show that the map $\Psi$ defined by Eq.\eqref{iafcom} is the unique nontrivial graded Hopf automorphism preserving the basis of fundamental
quasisymmetric functions.

\begin{proposition}\label{prop:fundambalg}
The maps $\Psi$, $\rho$ and $\omega$ are the only nontrivial graded algebra automorphisms of ${\rm {\Qsym}}$ that take the fundamental quasisymmetric basis into itself.
\end{proposition}

\begin{proof}
Let $\varphi$ be a nontrivial graded algebra automorphism of ${\rm {\Qsym}}$ such that $\varphi(\{F_{\alpha}|\alpha\in \mathcal{C}\})=\{F_{\alpha}|\alpha\in \mathcal{C}\}$.
From \cite[Th\'eor\`eme 4.12 and page 50]{M93} and \cite[Corollary 2.4]{MR95} we know that $\Psi, \rho$, and $\omega$ are graded algebra automorphisms of ${\rm {\Qsym}}$. So we only need to show that $\varphi\in\{\Psi, \rho,\omega\}$.

By Lemma \ref{lemma:alghomalbevaralbe},  $\varphi$ is not the identity map in degree less than $4$.
By Lemma \ref{lemma:identitymapgrdaed},  we have $\varphi(F_{1})=F_{1}$. Thus, $\varphi$ is not the identity map
at least in one of degree $2$ and $3$.
For degree $2$ there are two possibilities, i.e.,

Case $(1).$ $\varphi(F_{11})=F_{11}$, $\varphi(F_{2})=F_{2}$ and

Case $(2).$  $\varphi(F_{11})=F_{2}$, $\varphi(F_{2})=F_{11}$.

In the first case,  $\varphi$ is the identity map in degree $2$. So it is nontrivial in degree $3$.
By Eq.\eqref{fundbasf1falp},  we have
\begin{align}\label{eq:f1f2prod}
F_{1}F_{2}=F_{12}+F_{21}+F_{3}.
\end{align}
Since $\varphi$ is an algebra automorphism, we have
\begin{align*}
\langle \varphi(F_{111}),F_{1}F_{2}\rangle_F=\langle \varphi(F_{111}),\varphi(F_{1})\varphi(F_{2})\rangle_F=\langle F_{111},F_{1}F_{2}\rangle_F=0.
\end{align*}
Note that  $F_{111},F_{12},F_{21}$ and $F_{3}$ are the all fundamental quasisymmetric functions of degree $3$, and hence $\varphi(F_{111})=F_{111}$.
Similarly, from $\langle \varphi(F_{3}),F_{1}F_{11}\rangle_F=\langle F_{3},F_{1}F_{11}\rangle_F=0$
we conclude that $\varphi(F_{3})=F_{3}$.
But $\varphi$ is  not the identity map in degree $3$, so that we have $\varphi(F_{12})=F_{21}$ and $\varphi(F_{21})=F_{12}$.
This yields that $\varphi(F_{\alpha})=F_{\alpha^r}$ for all composition $\alpha$ with $|\alpha|\leq 3$.

In the second case, it follows from
$$\langle \varphi(F_{111}),F_{1}F_{11}\rangle_F= \langle \varphi(F_{111}),\varphi(F_{1})\varphi(F_{2})\rangle_F=\langle F_{111},F_{1}F_{2}\rangle_F=0$$
that $\varphi(F_{111})=F_{3}$. Similarly, from $\langle \varphi(F_{3}),F_{1}F_{2}\rangle_F=\langle F_{3},F_{1}F_{11}\rangle_F=0$
we conclude that $\varphi(F_{3})=F_{111}$. Therefore, for any $\alpha$ with $|\alpha|\leq 3$, we have $\varphi(F_\alpha)=F_{\alpha^c}$ if $\varphi(F_{12})=F_{21}$, $\varphi(F_{21})=F_{12}$
and have $\varphi(F_\alpha)=F_{\alpha^t}$  if $\varphi(F_{12})=F_{{12}}$, $\varphi(F_{21})=F_{21}$.

Applying Lemma \ref{lemma:alghomalbevaralbe}, we see that $\varphi\in\{\Psi, \rho,\omega\}$.
\end{proof}

Our next aim is to show that $\Psi$ is the unique nontrivial graded coalgebra automorphism of ${\rm {\Qsym}}$ that takes the fundamental quasisymmetric basis into itself.
To this end, we first investigate a connection between the concatenation and near concatenation of compositions.

Let $a$ be an integer and $B=\{b_1,b_2,\cdots,b_k\}$ a set of  integers. We define
$$
a+B=\{a+b_1,a+b_2,\cdots,a+b_k\}.
$$

\begin{lemma}\label{nearandconcatedes}
Let $\delta,\alpha,\beta$ be compositions with $|\delta|=|\alpha|+|\beta|$. Then
\begin{enumerate}
\item\label{detacdotalbeta} $\delta=\alpha\cdot \beta$ if and only if $\Set(\delta)=\Set(\alpha)\cup\{|\alpha|\} \cup (|\alpha|+\Set(\beta))$;

\item\label{detaodotalbetb}  $\delta=\alpha\odot \beta$ if and only if $\Set(\delta)=\Set(\alpha)\cup(|\alpha|+\Set(\beta))$.
\end{enumerate}
\end{lemma}
\begin{proof}
Let $\delta$ be the composition $(\delta_1,\delta_2,\cdots,\delta_k)$ with weight $m$.
If $\delta=\alpha\cdot \beta$, then there exists positive integer $i$ with $1\leq i\leq k-1$ such that
$\alpha= (\delta_1,\cdots,\delta_i)$, $\beta=(\delta_{i+1},\cdots,\delta_k)$. Thus,
$\Set(\delta)=\Set(\alpha)\cup\{|\alpha|\} \cup (|\alpha|+\Set(\beta))$.
If  $\delta=\alpha\odot \beta$, then there exists positive integers $a,b$ such that $\delta_{i}=a+b$ for some $i$ with $1\leq i\leq k$
and $\alpha=(\delta_1,\cdots,\delta_{i-1},a)$, $\beta=(b,\delta_{i+1},\cdots,\delta_k)$. Then $\Set(\delta)=\Set(\alpha)\cup(|\alpha|+\Set(\beta))$.

Conversely, let $\alpha\models p$, $\beta\models q$ with ${\rm Set}(\alpha)=\{a_1,a_2,\cdots,a_i\}$
and ${\rm Set}(\beta)=\{b_1,b_2,\cdots,b_j\}$. Thus, $p+q=m$.  If ${\rm Set}(\delta)={\rm Set}(\alpha)\cup\{|\alpha|\} \cup (|\alpha|+\Set(\beta))$,
then
\begin{align*}
{\rm Set}(\delta)=\{a_1,a_2,\cdots,a_i,p,p+b_1,p+b_2,\cdots,p+b_j\},
\end{align*}
and hence
\begin{align*}
\delta=(a_1,a_2-a_1,\cdots,a_i-a_{i-1},p-a_{i}, b_1,b_2-b_1,\cdots,b_{j}-b_{j-1},q-b_j)
=\alpha\cdot\beta.
\end{align*}
If $\Set(\delta)=\Set(\alpha)\cup(|\alpha|+\Set(\beta))$, then
\begin{align*}
{\rm Set}(\alpha)=\{a_1,a_2,\cdots,a_i,p+b_1,p+b_2,\cdots,p+b_j\},
\end{align*}
whence
\begin{align*}
\delta=(a_1,a_2-a_1,\cdots,a_i-a_{i-1},p-a_{i}+b_1,b_2-b_1,\cdots,b_{j}-b_{j-1},q-b_j)
=\alpha\odot\beta,
\end{align*}
as required.
\end{proof}

\begin{corollary}\label{nearandconcateeo}
Let $\alpha,\beta$ be two compositions. Then
$(\alpha\cdot \beta)^c=\alpha^c\odot\beta^c$, $(\alpha\odot \beta)^c=\alpha^c\cdot\beta^c$.
\end{corollary}
\begin{proof}
Let $\alpha\models p$, $\beta\models q$ for some nonnegative integers $p,q$, and let $\delta=\alpha\cdot \beta$.
By Lemma \ref{nearandconcatedes}\eqref{detacdotalbeta},
$\Set(\delta)=\Set(\alpha)\cup\{|\alpha|\} \cup (|\alpha|+\Set(\beta))$, and hence
\begin{align*}
[|\delta|-1]-\Set(\delta)=&[p+q-1]-\left(\Set(\alpha)\cup\{p\} \cup (p+\Set(\beta))\right)\cr
=&\left([p-1]-\Set(\alpha)\right)\cup \left([p+1,p+q-1]-(p+\Set(\beta))\right)\cr
=&\Set(\alpha^c)\cup \left(p+([q-1]-\Set(\beta))\right)\cr
=&\Set(\alpha^c)\cup\left(p+\Set(\beta^c)\right).
\end{align*}
It follows that $\Set(\delta^c)=\Set(\alpha^c)\cup\left(|\alpha^c|+\Set(\beta^c)\right)$, or equivalently,
$\delta^c=\alpha^c\odot\beta^c$ by Lemma \ref{nearandconcatedes}\eqref{detaodotalbetb}.

Note that for any nonnegative integer $n$ and any composition $\gamma\models n$, we have
\begin{align*}
\Set((\gamma^c)^c)=[n-1]-\Set(\gamma^c)=\Set(\gamma).
\end{align*}
It follows that $(\gamma^c)^c=\gamma$. Therefore, $(\alpha\odot \beta)^c=\alpha^c\cdot\beta^c$ is an immediate consequence of
$(\alpha\cdot \beta)^c=\alpha^c\odot\beta^c$.
\end{proof}

\begin{proposition}\label{fundambtwoco}
The map $\Psi$ is the unique nontrivial graded coalgebra automorphism of ${\rm {\Qsym}}$ that takes the fundamental quasisymmetric basis into itself.
\end{proposition}

\begin{proof}
By Eq.\eqref{qsprocop} and  Corollary \ref{nearandconcateeo}, for any composition $\alpha$, we have
\begin{align*}
(\Psi\otimes \Psi)\Delta(F_\alpha)
=&\sum_{\alpha=\beta\cdot\gamma}F_{\beta^c}\otimes F_{\gamma^c}+\sum_{\alpha=\beta\odot\gamma}F_{\beta^c}\otimes F_{\gamma^c}\cr
=&\sum_{\alpha^c=\beta^c\odot\gamma^c}F_{\beta^c}\otimes F_{\gamma^c}+\sum_{\alpha^c=\beta^c\cdot\gamma^c}F_{\beta^c}\otimes F_{\gamma^c}\cr
=&\Delta(\Psi(F_{\alpha})).
\end{align*}
Also, from Eq.\eqref{counitqs} one can easily check that $\varepsilon\Psi(F_\alpha)=\varepsilon(F_\alpha)$.
Thus, the map $\Psi$ is a coalgebra automorphism of ${\rm {\Qsym}}$.

Conversely, let $\psi$ be a  nontrivial graded coalgebra automorphism of ${\rm {\Qsym}}$ that preserves the fundamental quasisymmetric basis. We next show that
$\psi=\Psi$.

By Lemma \ref{lemma:identitymapgrdaed},  $\psi(F_1)=F_1$ and
 there are two possibilities for degree $2$, i.e.,

Case (1). $\psi(F_{11})=F_{11}$, $\psi(F_{2})=F_{2}$, and

Case (2). $\psi(F_{11})=F_{2}$, $\psi(F_{2})=F_{11}$.

Take an arbitrary composition $\alpha$ and  let $\psi(F_{\alpha})=F_{\delta}$.
We prove by induction on $|\alpha|$ that $\delta=\alpha$ if Case (1) happens, while
$\delta=\alpha^c$ if Case (2) happens.
The initial cases of $|\alpha|\leq2$ have been established.  Now suppose that the desired results hold for all $\alpha$ with $|\alpha|<n$ where $n\geq3$ and consider the case $|\alpha|=n$.
Then, by Eq.\eqref{qsprocop},
\begin{align}\label{phiph9idftoal}
(\psi\otimes\psi)\Delta(F_{\alpha})=&1\otimes F_{\delta}
+\sum_{{\alpha=\beta\cdot\gamma,}\atop{\beta,\gamma\neq\emptyset}}\psi(F_{\beta})\otimes \psi(F_{\gamma})
+\sum_{{\alpha=\beta\odot\gamma,}\atop{\beta,\gamma\neq\emptyset}}\psi(F_{\beta})\otimes \psi(F_{\gamma})+F_{\delta}\otimes 1.
\end{align}
Clearly, we have $|\beta|<|\alpha|=n$ and $|\gamma|<|\alpha|=n$ if $\alpha=\beta\cdot\gamma$ or $\alpha=\beta\odot\gamma$ with $\beta,\gamma\neq\emptyset$.

We first take care of Case (1). Assume that $\alpha\neq\delta$. By hypothesis, $\psi(F_\beta)=F_{\beta}$ for all compositions $\beta$ with $|\beta|<n$, so
\begin{align*}
(\psi\otimes\psi)\Delta(F_{\alpha})=1\otimes F_{\delta}
+\sum_{{\alpha=\beta\cdot\gamma,}\atop{\beta,\gamma\neq\emptyset}}F_{\beta}\otimes F_{\gamma}
+\sum_{{\alpha=\beta\odot\gamma,}\atop{\beta,\gamma\neq\emptyset}}F_{\beta}\otimes F_{\gamma}+F_{\delta}\otimes 1.
\end{align*}
Since $\psi$ is a coalgebra homomorphism, we have $(\psi\otimes\psi)\Delta(F_{\alpha})=\Delta(\psi(F_{\alpha}))=\Delta(F_{\delta})$, which together with Eq.\eqref{qsprocop} yields that
\begin{align}\label{deltasumntri1}
\sum_{{\alpha=\beta\cdot\gamma,}\atop{\beta,\gamma\neq\emptyset}}F_{\beta}\otimes F_{\gamma}
+\sum_{{\alpha=\beta\odot\gamma,}\atop{\beta,\gamma\neq\emptyset}}F_{\beta}\otimes F_{\gamma}
=\sum_{{\delta=\tau\cdot\sigma,}\atop{\tau,\sigma\neq\emptyset}}F_{\tau}\otimes F_{\sigma}
+\sum_{{\delta=\tau\odot\sigma,}\atop{\tau,\sigma\neq\emptyset}}F_{\tau}\otimes F_{\sigma}.
\end{align}
Denote
\begin{align*}
D_1=&\{F_{\beta}\otimes F_{\gamma}|\alpha=\beta\cdot\gamma\ {\rm and}\ \beta,\gamma\neq\emptyset\},
&D_2=&\{F_{\beta}\otimes F_{\gamma}|\alpha=\beta\odot\gamma\ {\rm and}\ \beta,\gamma\neq\emptyset\},\cr
E_1=&\{F_{\tau}\otimes F_{\sigma}|\delta=\tau\cdot\sigma\ {\rm and}\ \tau,\sigma\neq\emptyset\},
&E_2=&\{F_{\tau}\otimes F_{\sigma}|\delta=\tau\odot\sigma\ {\rm and}\ \tau,\sigma\neq\emptyset\}.
\end{align*}
By the definitions of the concatenation and near concatenation, $\alpha=\beta\cdot\gamma$ implies $\ell(\alpha)=\ell(\beta)+\ell(\gamma)$,
while $\alpha=\beta\odot\gamma$ implies $\ell(\alpha)=\ell(\beta)+\ell(\gamma)-1$, and hence
$D_1\cap D_2=\emptyset$ and $E_1\cap E_2=\emptyset$. Then, by Eq.\eqref{deltasumntri1}, we have the disjoint union $D_1\cup D_2=E_1\cup E_2$.
It follows from $\alpha\neq\delta$ that
$D_1\cap E_1=\emptyset$ and $D_2\cap E_2=\emptyset$, whence $D_1= E_2$ and $D_2= E_1$.

Note that $D_1= E_2\neq \emptyset$ implies that $\ell(\alpha)=\ell(\delta)+1$, and $D_2= E_1\neq \emptyset$ implies that $\ell(\alpha)=\ell(\delta)-1$,
so at least one of $D_1= E_2= \emptyset$ and $D_2= E_1=\emptyset$ holds. Without loss of generality, assume that $D_1= E_2= \emptyset$.
Now $D_1= \emptyset$ yields that $\alpha=(n)$, and $E_2= \emptyset$ yields that $\delta=(1^n)$. Since $n\geq3$, we have $\alpha=1\odot(n-1)$ and hence
$F_1\otimes F_{n-1}\in D_2=E_1$ so that $\delta=1\cdot(n-1)=(1,n-1)$, contradicting $n\geq3$.
Therefore, $\alpha=\delta$  proving that $\psi$ is the identity map.

Let us consider Case (2). By hypothesis, $\psi(F_{\beta})=F_{\beta^c}$ for all compositions $\beta$ with $|\beta|<n$, $n\geq3$.
Thus, $(\psi\otimes\psi)\Delta(F_{\alpha})=\Delta(\psi(F_{\alpha}))=\Delta(F_{\delta})$, together with Eq.\eqref{phiph9idftoal}, implies that
\begin{align}\label{deltasumnphiuntri}
\sum_{{\alpha=\beta\cdot\gamma,}\atop{\beta,\gamma\neq\emptyset}}F_{\beta^c}\otimes F_{\gamma^c}
+\sum_{{\alpha=\beta\odot\gamma,}\atop{\beta,\gamma\neq\emptyset}}F_{\beta^c}\otimes F_{\gamma^c}
=\sum_{{\delta=\tau\cdot\sigma,}\atop{\tau,\sigma\neq\emptyset}}F_{\tau}\otimes F_{\sigma}
+\sum_{{\delta=\tau\odot\sigma,}\atop{\tau,\sigma\neq\emptyset}}F_{\tau}\otimes F_{\sigma}.
\end{align}
By Corollary \ref{nearandconcateeo}, the left-hand side of Eq.\eqref{deltasumnphiuntri} is equal to
\begin{align}\label{deltasumnphiuntri2}
\sum_{{\alpha^c=\beta^c\odot\gamma^c,}\atop{\beta,\gamma\neq\emptyset}}F_{\beta^c}\otimes F_{\gamma^c}
+\sum_{{\alpha^c=\beta^c\cdot\gamma^c,}\atop{\beta,\gamma\neq\emptyset}}F_{\beta^c}\otimes F_{\gamma^c}.
\end{align}
Let $\zeta=\beta^c$ and $\eta=\gamma^c$. Then combining \eqref{deltasumnphiuntri} and \eqref{deltasumnphiuntri2} we obtain that
\begin{align*}
\sum_{{\alpha^c=\zeta\odot\eta,}\atop{\zeta,\eta\neq\emptyset}}F_{\zeta}\otimes F_{\eta}
+\sum_{{\alpha^c=\zeta\cdot\eta,}\atop{\zeta,\eta\neq\emptyset}}F_{\zeta}\otimes F_{\eta}
=\sum_{{\delta=\tau\cdot\sigma,}\atop{\tau,\sigma\neq\emptyset}}F_{\tau}\otimes F_{\sigma}
+\sum_{{\delta=\tau\odot\sigma,}\atop{\tau,\sigma\neq\emptyset}}F_{\tau}\otimes F_{\sigma},
\end{align*}
which is a formula analogous to Eq.\eqref{deltasumntri1}.
Thus in analogy to the proof of Case (1) we obtain that $\delta=\alpha^c$. Therefore, $\psi(F_\alpha)=F_{\alpha^c}$ for all
compositions $\alpha$, and hence $\psi=\Psi$. This completes the proof.
\end{proof}

\begin{theorem}\label{fundambtwohopfalg}
The map $\Psi$ is the only nontrivial graded Hopf algebra automorphism of ${\rm {\Qsym}}$ that preserves the fundamental quasisymmetric basis.
\end{theorem}
\begin{proof}
By Propositions \ref{prop:fundambalg} and \ref{fundambtwoco}, $\Psi$ is the only nontrivial graded algebra and coalgebra automorphism of ${\rm {\Qsym}}$ that preserves the fundamental quasisymmetric basis, which establishes the desired statement.
\end{proof}


\section{Rigidity with respect to the quasisymmetric Schur basis}\label{sec:rwrttqsf}

Now we prove that ${\rm {\Qsym}}$ is rigid respectively as a graded algebra and coalgebra with respect to the quasisymmetric Schur basis.
Before turning to the main results, we first  prove two special Pieri-type formulas for quasisymmetric Schur functions.
\begin{lemma}\label{lem:s2progucts12121}
Let $i_1,\cdots,i_k, j_1,\cdots,j_k$ be nonnegative integers. Then
\begin{enumerate}
\item\label{eq:items2s1i12j112exp}
${\mathcal{S}}_{2}{\mathcal{S}}_{(1^{i_1},2^{j_1},\cdots,1^{i_k},2^{j_k},1,2)}=\sum_{\beta}{\mathcal{S}}_{\beta}$ contains $3(j_1+\cdots+j_k)+4$ terms;
\item\label{eq:items2s1i12j121exp}
${\mathcal{S}}_{2}{\mathcal{S}}_{(1^{i_1},2^{j_1},\cdots,1^{i_k},2^{j_k+1},1)}=\sum_{\beta}{\mathcal{S}}_{\beta}$ contains $3(j_1+\cdots+j_k)+6$ terms.
\end{enumerate}
\end{lemma}
\begin{proof}
\eqref{eq:items2s1i12j112exp}
Let $\alpha=(1^{i_1},2^{j_1},\cdots,1^{i_k},2^{j_k},1,2)$. Then $\widetilde{\alpha}=(2^{j_1+\cdots+j_k+1},1^{i_1+\cdots+i_k+1})$. Suppose that ${\mathcal{S}}_{\beta}$ is a term of the product ${\mathcal{S}}_{2}{\mathcal{S}}_{\alpha}$. Then, by Lemma \ref{PieriruqSchurf}\eqref{eq:itemPieriruleqsfsnsal}, $\delta=\widetilde{\beta}/\widetilde{\alpha}$ is a horizontal strip of size $2$.
There are four cases  for $\beta$:

{\emph{Case 1.}} $\widetilde{\beta}=(2^{j_1+\cdots+j_k+2},1^{i_1+\cdots+i_k+1})$. Then $S(\delta)=\{1,2\}$. It follows from $\text{row}_{S(\delta)}(\beta)=\alpha$ that
$\beta=(1^{i_1},2^{j_1},\cdots,1^{i_k},2^{j_k},1,2,2)$.

{\emph{Case 2.}} $\widetilde{\beta}=(3,2^{j_1+\cdots+j_k},1^{i_1+\cdots+i_k+2})$. Then  $S(\delta)=\{1,3\}$ and hence
$\beta$ is obtained from $\alpha$ by adding a part of size $1$ near the last part of size $2$ and then replacing a part of size $2$ with a part of size $3$. It will contribute $2(j_1+\cdots+j_k+1)$ terms to ${\mathcal{S}}_{2}{\mathcal{S}}_{(1^{i_1},2^{j_1},\cdots,1^{i_k},2^{j_k},1,2)}$.

{\emph{Case 3.}} $\widetilde{\beta}=(3,2^{j_1+\cdots+j_k+1},1^{i_1+\cdots+i_k})$.  Then $S(\delta)=\{2,3\}$ and hence
 there is no $\beta$ such that $\text{row}_{S(\delta)}(\beta)=\alpha$. Otherwise, the size of each part of such a $\beta$ is less than $4$, so the size of the last part
 of $\text{row}_{S(\delta)}(\beta)$ must be $1$, that is,  $\text{row}_{S(\delta)}(\beta)\neq\alpha$, a contradiction.
Hence no such a composition $\beta$ exists.

{\emph{Case 4.}} $\widetilde{\beta}=(4,2^{j_1+\cdots+j_k},1^{i_1+\cdots+i_k+1})$. Then $S(\delta)=\{3,4\}$ and hence
$\text{row}_{S(\delta)}(\beta)=\alpha$ yields that $\beta$ is obtained from
$\alpha$ by replacing a part of size $2$ with a part of size $4$. It will contribute $j_1+\cdots+j_k+1$ terms to ${\mathcal{S}}_{2}{\mathcal{S}}_{(1^{i_1},2^{j_1},\cdots,1^{i_k},2^{j_k},1,2)}$.
This proves part \eqref{eq:items2s1i12j112exp}.

\eqref{eq:items2s1i12j121exp} This part can be proved similarly. Let $\sigma=(1^{i_1},2^{j_1},\cdots,1^{i_k},2^{j_k+1},1)$ and
 suppose that ${\mathcal{S}}_{\beta}$ is a term of the product ${\mathcal{S}}_{2}{\mathcal{S}}_{\sigma}$.
 Then $\widetilde{\sigma}=(2^{j_1+\cdots+j_k+1},1^{i_1+\cdots+i_k+1})$ and
 $\delta=\widetilde{\beta}/\widetilde{\sigma}$ is a horizontal strip of size $2$, so there are also four possibilities for $\beta$:

{\emph{Case 1.}} $\widetilde{\beta}=(2^{j_1+\cdots+j_k+2},1^{i_1+\cdots+i_k+1})$. Then $S(\delta)=\{1,2\}$ and hence
$$\beta=(1^{i_1},2^{j_1},\cdots,1^{i_k},2^{j_k+1},1,2),\ {\rm or}\ \beta=(1^{i_1},2^{j_1},\cdots,1^{i_k},2^{j_k+2},1).$$

{\emph{Case 2.}} $\widetilde{\beta}=(3,2^{j_1+\cdots+j_k},1^{i_1+\cdots+i_k+2})$. Then  $S(\delta)=\{1,3\}$ and hence
$\beta$ is obtained from
$\sigma$ by replacing a part of size $2$ with a part of size $3$ and suffixing a part of size $1$.
There are $j_1+\cdots+j_k+1$ possible choices for $\beta$.

{\emph{Case 3.}} $\widetilde{\beta}=(3,2^{j_1+\cdots+j_k+1},1^{i_1+\cdots+i_k})$.  Then $S(\delta)=\{2,3\}$ and hence
$\beta$ is obtained from $\sigma$ either by  replacing the last part of size $1$ with a part of size $3$, or by replacing a part of size $2$ with a part of size $3$ and replacing the last part of size $1$ with a part of size $2$. So
there are $j_1+\cdots+j_k+2$ possible choices for $\beta$.

{\emph{Case 4.}} $\widetilde{\beta}=(4,2^{j_1+\cdots+j_k},1^{i_1+\cdots+i_k+1})$. Then $S(\delta)=\{3,4\}$ and it is easy to see that
$\beta$ is obtained from
$\sigma$ by replacing a part of size $2$ with a part of size $4$. This contributes $j_1+\cdots+j_k+1$ terms to ${\mathcal{S}}_{2}{\mathcal{S}}_{(1^{i_1},2^{j_1},\cdots,1^{i_k},2^{j_k+1},1)}$.

Therefore, ${\mathcal{S}}_{2}{\mathcal{S}}_{(1^{i_1},2^{j_1},\cdots,1^{i_k},2^{j_k+1},1)}=\sum_{\beta}{\mathcal{S}}_{\beta}$ contains $3(j_1+\cdots+j_k)+6$ terms.
\end{proof}

As an example for Lemma \ref{lem:s2progucts12121}, for any positive integer $i$ we have
\begin{align*}
{\mathcal{S}}_{2}{\mathcal{S}}_{(1^i,2)}=&{\mathcal{S}}_{(1^i,4)}+{\mathcal{S}}_{(1^{i+1},3)}+{\mathcal{S}}_{(1^{i},3,1)}
+{\mathcal{S}}_{(1^{i},2,2)},\\
{\mathcal{S}}_{2}{\mathcal{S}}_{(1^{i},2,1)}=&{\mathcal{S}}_{(1^{i},4,1)}+{\mathcal{S}}_{(1^{i},3,1,1)}+{\mathcal{S}}_{(1^{i},3,2)}
+{\mathcal{S}}_{(1^{i},2,3)}+{\mathcal{S}}_{(1^{i},2,2,1)}
+{\mathcal{S}}_{(1^{i},2,1,2)}.
\end{align*}

\begin{proposition}\label{thm:alhomqsfb}
The identity map is the only graded algebra automorphism of ${\rm {\Qsym}}$ that preserves the quasisymmetric Schur basis.
\end{proposition}
\begin{proof}
Let $\varphi$ be a graded algebra automorphism of ${\rm {\Qsym}}$ that preserves the quasisymmetric Schur basis.
We shall show that $\varphi({\mathcal{S}}_{\alpha})={\mathcal{S}}_{\alpha}$ for all compositions $\alpha$ by induction on the weight $|\alpha|$.
By Lemma \ref{lemma:identitymapgrdaed}, we have $\varphi({\mathcal{S}}_{1})={\mathcal{S}}_{1}$.
For the component of degree $2$, we claim that $\varphi$ is also the identity map. Otherwise, Lemma \ref{lemma:identitymapgrdaed} guarantees that
$\varphi({\mathcal{S}}_{11})={\mathcal{S}}_{2}$ and $\varphi({\mathcal{S}}_{2})={\mathcal{S}}_{11}$,
so that $\varphi({\mathcal{S}}_{11}{\mathcal{S}}_{11})={\mathcal{S}}_{2}{\mathcal{S}}_{2}$, which means that ${\mathcal{S}}_{11}{\mathcal{S}}_{11}$ and ${\mathcal{S}}_{2}{\mathcal{S}}_{2}$ have the same number of terms when they are written as a linear combination of quasisymmetric   Schur functions, contradicting Example \ref{lem:s11s11proguctss2s2}.

Now assume that $|\alpha|\geq3$, and that $\varphi$ is the identity map on the components of degree $|\alpha|-1$.
Let $\beta$ be the composition such that $\varphi({\mathcal{S}}_\alpha)={\mathcal{S}}_\beta$. Then $|\alpha|=|\beta|$.
By the induction hypothesis, we have
\begin{align}\label{eq:qsyalbeprodexpqsy}
\langle{\mathcal{S}}_\beta,{\mathcal{S}}_{1}{\mathcal{S}}_\sigma\rangle_{\mathcal{S}}
=\langle\varphi({\mathcal{S}}_\alpha),\varphi({\mathcal{S}}_1)\varphi({\mathcal{S}}_\sigma)\rangle_{\mathcal{S}}
=\langle {\mathcal{S}}_\alpha,{\mathcal{S}}_1 {\mathcal{S}}_\sigma\rangle_{\mathcal{S}}
\end{align}
for all compositions $\sigma$ with $|\sigma|=|\alpha|-1$.
By Eq.\eqref{qsymschurbass1salp}, we have
$D_{\leq_Q}(\alpha)=D_{\leq_Q}(\beta)$. If $\alpha\neq\beta$, then, by Proposition \ref{lem:aln=beccoversetleqq}, we may assume without loss of generality that
there exist nonnegative integers $i_1,\cdots,i_k, j_1,\cdots,j_k$ where $k\geq0$ such that
$$\alpha=(1^{i_1},2^{j_1},\cdots,1^{i_k},2^{j_k},1,2),\quad \beta=(1^{i_1},2^{j_1},\cdots,1^{i_k},2^{j_k+1},1).$$
Since $\varphi$ is multiplication-preserving, we have  $\varphi({\mathcal{S}}_2{\mathcal{S}}_\alpha)
=\varphi({\mathcal{S}}_2)\varphi({\mathcal{S}}_\alpha)
={\mathcal{S}}_2{\mathcal{S}}_\beta$, contradicting Lemma \ref{lem:s2progucts12121}. Thus $\alpha=\beta$, as required.
\end{proof}

\begin{lemma}\label{lem:qsfidvarphis}
Let $f$ be a graded coalgebra automorphism that preserves the quasisymmetric Schur basis.
If  $f({\mathcal{S}}_{11})={\mathcal{S}}_{11}$, $f({\mathcal{S}}_{2})={\mathcal{S}}_{2}$,
then $f$ is the identity map.
\end{lemma}

\begin{proof}
It suffices to show that $f({\mathcal{S}}_{\gamma})={\mathcal{S}}_{\gamma}$ for all compositions $\gamma$.
We proceed by induction on $|\gamma|$. If $|\gamma|\leq 2$, the result is immediate.  Now let $n=|\gamma|\geq3$, and assume that
for all $\alpha$ with $|\alpha|<n$, we have $f({\mathcal{S}}_{\alpha})={\mathcal{S}}_{\alpha}$.

Suppose to the contrary that $f({\mathcal{S}}_{\gamma})\neq{\mathcal{S}}_{\gamma}$. Then there exists a composition $\delta$ of weight
$n$ such that $\delta\neq\gamma$ and $f({\mathcal{S}}_{\gamma})={\mathcal{S}}_{\delta}$. Thus, by the induction hypothesis, together with Eq.\eqref{eq:deltasagammacoprod},
we have
\begin{align*}
(f\otimes f)\Delta({\mathcal{S}}_{\gamma})=\sum_{\alpha,\beta}{C}_{\alpha\beta}^\gamma f({\mathcal{S}}_{\alpha})\otimes f({\mathcal{S}}_{\beta})
=1\otimes {\mathcal{S}}_{\delta}
+\sum_{{\alpha\neq\emptyset,\beta\neq\emptyset}}{C}_{\alpha\beta}^\gamma {\mathcal{S}}_{\alpha}\otimes {\mathcal{S}}_{\beta}
+{\mathcal{S}}_{\delta}\otimes1
\end{align*}
where the sum is over all compositions $\alpha,\beta$.
On the other hand,
\begin{align*}
\Delta(f({\mathcal{S}}_{\gamma}))=\Delta({\mathcal{S}}_{\delta})=1\otimes {\mathcal{S}}_{\delta}
+\sum_{{\alpha\neq\emptyset,\beta\neq\emptyset}}{C}_{\alpha\beta}^\delta {\mathcal{S}}_{\alpha}\otimes {\mathcal{S}}_{\beta}
+{\mathcal{S}}_{\delta}\otimes1.
\end{align*}
By comparing coefficients we obtain that ${C}_{\alpha\beta}^{\gamma}={C}_{\alpha\beta}^{\delta}$ for all
$\alpha,\beta$ with $\alpha\neq\emptyset$, $\beta\neq\emptyset$.

By Eq.\eqref{eq:nlrcnvsp},
 $\beta\prec_C\gamma$ is equivalent to the condition that ${C}_{1,\beta}^{\gamma}=1$, whence  $D_{\leq_{C}}(\gamma)=D_{\leq_{C}}(\delta)$.
Note that $n\geq3$. So, without loss of generality, by Proposition \ref{lembgdcovers},
we have the following two cases:

{\emph{Case 1.}} $\delta=(3)$, $\gamma=(1,2)$. Let $\beta=(1)$. Then $\beta\leq_{C}\delta$ and $\beta\leq_{C}\gamma$.
Note that, as shown below,
\begin{center}
\begin{tabular}{c}
$\delta/\hspace{-1.2mm}/\beta=$\ \begin{ytableau}
\bullet & & \\
\end{ytableau}
\end{tabular}
\qquad
\begin{tabular}{c}
$\gamma/\hspace{-1.2mm}/\beta=$\ \begin{ytableau}
  \\
\bullet & \\
\end{ytableau}
\end{tabular}
\end{center}
the skew reverse composition shape $\gamma/\hspace{-1.2mm}/\beta$ is a vertical strip, while $\delta/\hspace{-1.2mm}/\beta$ is not.
So, by Eq.\eqref{eq:nlrcnvsp}, ${C}_{11,\beta}^{\delta}=0$ while ${C}_{11,\beta}^{\gamma}=1$,
contradicting ${C}_{\alpha\beta}^{\gamma}={C}_{\alpha\beta}^{\delta}$ for all
$\alpha,\beta$ with $\alpha\neq\emptyset$, $\beta\neq\emptyset$.

{\emph{Case 2.}} $\delta=(a+1,1,b_1,\cdots,b_k)$, $\gamma=(1,a,1,b_1,\cdots,b_k)$ where $a\in\{1,2\}$, $k\geq0$ and  $b_1,b_2,\cdots,b_k$ are positive integers satisfying
\begin{align*}
b_{j}\leq \max\{b_{i}+1|i=0,1,\cdots,j-1\}
\end{align*}
for all $1\leq j\leq k$,
with the notation $b_0=1$.
Let
\begin{align*}
\beta=\begin{cases}
(1,b_1,\cdots,b_k)&   if\ a=1,\\
(1,1,b_1,\cdots,b_k)& if\ a=2.
\end{cases}
\end{align*}
Then,  we have $\beta\leq_{C}\delta$, $\beta\leq_{C}\gamma$.
Note that $\gamma/\hspace{-1.2mm}/\beta$ is a vertical strip, while $\delta/\hspace{-1.2mm}/\beta$ is not.
Thus, ${C}_{11,\beta}^{\delta}=0$, ${C}_{11,\beta}^{\gamma}=1$, again a contradiction.

Therefore, $f({\mathcal{S}}_{\gamma})={\mathcal{S}}_{\gamma}$ for all compositions $\gamma$,
and the proof follows.
\end{proof}

In order to show the rigidity for ${\rm {\Qsym}}$ as a coalgebra with respect to the quasisymmetric Schur basis, we need the following lemma.
\begin{lemma}{\rm$($Corollary 6.9 in \cite{HLMW11}$)$}\label{lem:s=fquasisy}
Let $\alpha$ be a composition. Then $\mathcal{{S}}_{\alpha}=F_{\alpha}$ if and only if
$\alpha=(m, 1^{e_1},2,1^{e_2},\cdots,2,1^f)$ where $m,f,e_i$ are nonnegative integers such that $m\neq1$, $f\geq0$ and $e_i\geq1$ for all $i$.
\end{lemma}

\begin{lemma}\label{lem:nomapqsf}
There is no graded coalgebra automorphism $f$ preserving the quasisymmetric Schur basis such that $f({\mathcal{S}}_{11})={\mathcal{S}}_{2}$,
$f({\mathcal{S}}_{2})={\mathcal{S}}_{11}$.
\end{lemma}
\begin{proof}
Let $f$ be a graded coalgebra automorphism that preserves the quasisymmetric Schur basis such that $f({\mathcal{S}}_{11})={\mathcal{S}}_{2}$,
$f({\mathcal{S}}_{2})={\mathcal{S}}_{11}$. It follows from Lemma \ref{lem:s=fquasisy} that ${\mathcal{S}}_{\alpha}=F_{\alpha}$ for all compositions $\alpha$ with $|\alpha|\leq3$.
Thus, by Eq.\eqref{qsprocop}, one has
\begin{align*}
\Delta({\mathcal{S}}_{3})=\Delta(F_{3})
=&1\otimes F_{3}+F_{1}\otimes F_{2}+F_{2}\otimes F_{1}+F_{3}\otimes 1\cr
=&1\otimes {\mathcal{S}}_{3}+{\mathcal{S}}_{1}\otimes {\mathcal{S}}_{2}
+{\mathcal{S}}_{2}\otimes {\mathcal{S}}_{1}+{\mathcal{S}}_{3}\otimes 1.
\end{align*}
If $f({\mathcal{S}}_{3})={\mathcal{S}}_{\sigma}$ for some $\sigma\models 3$, then
we have $F_{\sigma}={\mathcal{S}}_{\sigma}$ and hence
\begin{align*}
\Delta(F_{\sigma})=&\Delta({\mathcal{S}}_{\sigma})=\Delta(f({\mathcal{S}}_{3}))=(f\otimes f)\Delta({\mathcal{S}}_{3})\cr
=&1\otimes {\mathcal{S}}_{\sigma}+{\mathcal{S}}_{1}\otimes {\mathcal{S}}_{11}
+{\mathcal{S}}_{11}\otimes {\mathcal{S}}_{1}+{\mathcal{S}}_{\sigma}\otimes 1\cr
=&1\otimes F_{\sigma}+F_{1}\otimes F_{11}+F_{11}\otimes F_{1}+F_{\sigma}\otimes 1.
\end{align*}
Thus, $\sigma=(1,1,1)$ by Eq.\eqref{qsprocop}, that is, $f({\mathcal{S}}_{3})={\mathcal{S}}_{111}$.
Similarly, we can obtain that $f({\mathcal{S}}_{111})={\mathcal{S}}_{3}$, $f({\mathcal{S}}_{12})={\mathcal{S}}_{21}$
and $f({\mathcal{S}}_{21})={\mathcal{S}}_{12}$.
Consequently, $f({\mathcal{S}}_{\gamma})={\mathcal{S}}_{\gamma^c}$ for all compositions $\gamma$ with $|\gamma|\leq3$.

Let $\mu$ be any composition of $4$, and set $f({\mathcal{S}}_{\mu})={\mathcal{S}}_{\delta}$ for some $\delta$ with $|\delta|=4$.
Directly comparing coefficients at the summands $\mathcal{S}_{1}\otimes {\mathcal{S}}_{\beta}$ for $|\beta|=3$ on both sides of the equation
$$
\Delta({\mathcal{S}}_{\delta})=\Delta(f({\mathcal{S}}_{\mu}))=(f\otimes f)\Delta({\mathcal{S}}_{\mu}),
$$
we see from Eq.\eqref{eq:deltasagammacoprod} that ${C}_{1,\beta}^{\delta}={C}_{1,\beta^c}^{\mu}$ for all compositions $\beta$ of weight $3$.
Then Eq.\eqref{eq:nlrcnvsp} gives the contradiction for $\mu=(1,3)$ as ${C}_{1,3}^{13}={C}_{1,12}^{13}=1$, but ${C}_{1,111}^{\delta}={C}_{1,21}^{\delta}=1$ for no $\delta$.
This completes the proof.
\end{proof}

\begin{proposition}\label{prop:qsfcoalg}
The identity map is the only graded coalgebra automorphism of ${\rm {\Qsym}}$ that preserves the set of quasisymmetric Schur functions.
\end{proposition}
\begin{proof}
By Lemma \ref{lem:nomapqsf}, any graded coalgebra automorphism of ${\rm {\Qsym}}$ must be the identity map in degree $2$,
which together with Lemma \ref{lem:qsfidvarphis} gives the desired conclusion.
\end{proof}

From Proposition \ref{thm:alhomqsfb} or Proposition \ref{prop:qsfcoalg}, we obtain the following result.
\begin{theorem}
${\rm {\Qsym}}$ is rigid as a Hopf algebra with respect to the quasisymmetric Schur basis, i.e.,
there are no nontrivial graded Hopf algebra automorphisms that take the set of quasisymmetric Schur functions into itself.
\end{theorem}


\noindent

{\bf Acknowledgements.}
This work was partially supported by  the Fundamental Research Funds for the Central Universities $($Grant No. XDJK2019C049$)$ and National Natural Science Foundation of China
$($Grant No. 11501467$)$.

%

\end{document}